\documentclass[twocolumn]{autart}

\usepackage{graphicx}
\usepackage{caption}
\usepackage[T1]{fontenc}

\usepackage{amsmath,amssymb}
\usepackage{amsfonts}
\usepackage{centernot}
\usepackage{enumitem}

\newtheorem{Theorem}{Theorem}[section]
\newtheorem{Proposition}[Theorem]{Proposition}

\newtheorem{Corollary}[Theorem]{Corollary}
\newtheorem{Lemma}[Theorem]{Lemma}
\newtheorem{Definition}[Theorem]{Definition}
\newtheorem{Remark}[Theorem]{Remark}

\newtheorem{Assumption}[Theorem]{Assumption}

\usepackage[usenames,dvipsnames]{xcolor}
\definecolor{wheat}{rgb}{0.96,0.87,0.70}
\definecolor{myGreen}{rgb}{0.,0.8,0.0}

\definecolor{darkgreen}{rgb}{0,0.6,0}

\makeatletter

\newcommand{\circleddagger}{{\mathbin{\mathpalette\circleddagger@{}}}}
\newcommand{\circleddagger@}[2]{%
	\ooalign{%
		\hfil$#1\dagger$\hfil\cr
		$#1\bigcirc$\cr
	}%
}
\makeatother

\newcommand {\matr}[2]{\left[\begin{array}{#1}#2\end{array}\right]}

\newcounter{lastnote}

\begin{document} 

\begin{frontmatter}

\title{Rethinking Strict Dissipativity for Economic MPC\thanksref{footnoteinfo}}
\thanks[footnoteinfo]{This paper was not presented at any IFAC 
	meeting. Corresponding author M. Zanon.}
	
\author[Mario]{Mario Zanon}\ead{mario.zanon@imtlucca.it}
\address[Mario]{IMT School for Advanced Studies Lucca, Piazza San Francesco 19, 55100, Lucca, Italy}

\begin{abstract}
	Stability of economic model predictive control can be proven under the assumption that a strict dissipativity condition holds. This assumption has a clear interpretation in terms of the so-called rotated stage cost, which must have its minimum at the optimal steady state. However, contrary to dissipativity, for strict dissipativity the storage function cannot be immediately related to the value function of an optimal control problem formulated with the economic stage cost. We propose the novel concept of two-storage strict dissipativity, which requires two storage functions to satisfy dissipativity and be separated by a positive definite function. This new condition can be immediately related to optimal control by means of value functions and might be easier to verify than strict dissipativity. Furthermore, we prove that two-storage strict dissipativity is sufficient and necessary for asymptotic stability, it is related to strict dissipativity, and also to alternative approaches relying on the so-called cost-to-travel. Finally, we discuss commonly used and new terminal cost designs that guarantee asymptotic stability in the finite-horizon case.
\end{abstract}


\end{frontmatter}

\section{Introduction}

Model Predictive Control (MPC) is a control technique in which a finite-horizon Optimal Control Problem (OCP) is solved at each time step. The first optimal control input is applied, the next state is measured, and the OCP is solved again in order to update the optimal control input in a receding horizon fashion~\cite{Rawlings2017,Gruene2017}.

In tracking MPC, stability is proven under the assumption that the cost function has its minimum at the optimal steady state. On the contrary, in economic MPC (EMPC) the cost is generic. While this aims at improving closed-loop performance, providing asymptotic stability guarantees becomes much harder and typically requires a strict dissipativity assumption to hold~\cite{Faulwasser2018a}. Early results include~\cite{Diehl2011,Amrit2011a,Angeli2012a}, where the main idea was first presented. Necessity of dissipativity for optimal steady-state operation has been proven in~\cite{Mueller2015,Mueller2015a} under some technical controllability assumption. Stability in the absence of terminal constraints has been analyzed in~\cite{Grune2013a,Zanon2018a,Faulwasser2018,Zanon2025,Gruene2025,Zanon2025a}. The periodic case has been analyzed in~\cite{Zanon2017e,Mueller2016} using the strict dissipativity framework and in~\cite{Houska2015,Houska2017} using a different approach. The discounted case has been studied in~\cite{Gaitsgory2018,Zanon2022a}. Stability results for stochastic MPC are provided in~\cite{Gros2022}, while the stability analysis of economic MPC in a reinforcement learning context is given in~\cite{Zanon2022b,Gros2022a}. Economic MPC in general requires one to solve harder optimization problems than tracking MPC. 
Computationally efficient algorithms tailored to EMPC have been proposed in~\cite{Quirynen2016,Quirynen2017a,Verschueren2017,Zanon2021b}. Though extremely relevant, this aspect is outside the scope of this paper, and we will focus instead on stability theory and on further connecting (strict) dissipativity and optimal control.

The connection between dissipativity and optimal control has been discussed already in the seminal papers on dissipativity in continuous time~\cite{Willems1971,Willems1972a,Willems1972b}, 
though strict dissipativity was introduced for economic MPC only much later~\cite{Diehl2011,Amrit2011a,Angeli2012a}. In continuous time, \cite{Faulwasser2021} relates (strict) dissipativity of infinite-horizon OCPs to the finiteness of the value function and to the convergence of optimal primal and adjoint trajectories, while strict dissipativity is the key assumption behind turnpike results such as~\cite{Trelat2023}. 
A discrete-time adaptation of the seminal papers can be found in~\cite{Lopezlena2006}. 
While strict dissipativity allows one to prove asymptotic stability, as we discuss below, choosing the storage function as the negative of the forward infinite-horizon value function does not yield strict dissipativity, since the rotated stage cost vanishes along optimal trajectories.
Instead, strict dissipativity unveils a hidden tracking MPC problem, as the so-called rotated cost does have its minimum at the optimal steady state.

In this paper, we introduce a new strict dissipativity concept: \emph{two-storage strict dissipativity}, which relies on two storage functions, each satisfying dissipativity, and their difference being lower bounded by a positive definite function. This last requirement has an immediate interpretation in terms of value functions of two OCPs formulated forward and backward in time, whose value function can only be equal at the optimal steady state. The idea of using two storage functions satisfying dissipativity in order to prove asymptotic stability has been first proposed in~\cite{Olanrewaju2017} for the linear-quadratic case, and further elaborated on and investigated in depth in~\cite{Zanon2025a} for the linear-quadratic case, where it is also proven that in that setting it is fully equivalent to strict dissipativity. We extend the ideas of~\cite{Olanrewaju2017,Zanon2025a} from the linear-quadratic setting to the general constrained nonlinear setting, and we establish a deep connection with the original dissipativity theory of Willems, suitably modified. We prove both sufficiency and necessity for asymptotic stability, we discuss the relation with strict dissipativity and with value functions of suitably defined OCPs. The newly introduced condition might be easier to check than standard strict dissipativity. 
Furthermore, we also analyze the finite-horizon setting, both using the terminal conditions first proposed in~\cite{Amrit2011a}, and with relaxed terminal conditions, under the assumption that the prediction horizon be long enough.

This paper is structured as follows. We provide the definitions and foundations of our work in Section~\ref{sec:problem_statement}. We discuss sufficiency and necessity of two-storage strict dissipativity in Section~\ref{sec:infinite_horizon} for the infinite horizon case. In Section~\ref{sec:finite_horizon} we discuss the finite-horizon case, and in Section~\ref{sec:cost-to-travel} we outline similarities and differences with the approach proposed in~\cite{Houska2015,Houska2017}. We provide two numerical examples in Section~\ref{sec:examples}, and conclude the paper with Section~\ref{sec:conclusions}.

\section{Problem Statement}
\label{sec:problem_statement}

Consider the system
\begin{align}
	\label{eq:real_system}
	x_{k+1} &= f(x_k, u_k),
\end{align}
with state $x\in\mathbb{R}^{n_x}$ and control input $u\in\mathbb{R}^{n_u}$, where $k$ denotes the time and $x_k$, $u_k$ are, respectively, the state and input at time $k$. The system should be operated such that constraints $h(x,u)\leq 0$ be satisfied at all times and stage cost $\ell(x,u)$ be minimized. We define the constraint set $\mathcal{H}:=\{\, (x,u) \,|\, h(x,u)\leq 0 \,\}$ and its projection $\mathcal{X}_h := \{\, x \, | \, \exists \, u  \,\text{ s.t. } \, h(x,u)\leq 0 \,\}$.

	In this paper we restrict our attention to the characterization of cases in which the system is optimally operated at steady state, according to the following definition.
	\begin{Definition}
		\label{def:optimal_steady_state}
		System~\eqref{eq:real_system} is \emph{optimally operated at steady state}, if for each $x_0$, $u_k$,  $x_{k+1}=f(x_k,u_k)$ such that $(x_k,u_k)\in\mathcal{H}$ we have $\liminf_{N\to\infty} \sum_{k=0}^{N-1} \frac{\ell(x_k,u_k)}{N}\geq \ell(x_\mathrm{s},u_\mathrm{s})$.
	\end{Definition}
We assume without loss of generality that the optimal steady state is the origin, i.e., $(x_\mathrm{s}, u_\mathrm{s})=(0,0)$. We will formalize this assumption, together with a few other requirements in Assumption~\ref{ass:regularity}.

The objective to control the system so as to minimize the stage cost while satisfying the constraints can be achieved by formulating and solving an optimal control problem. While this approach guarantees to deliver the two aforementioned objectives by design, it does not give any a-priori guarantee on the closed-loop behavior. Consequently, one of the main concerns in the literature on Economic MPC is to characterize under which conditions asymptotic stability is obtained. 

We write $\mathbb{I}_a^b:=\left \{y\in\mathbb{Z} \, \mid\, a\leq y \leq b \right \}$. Economic MPC is based on the OCP
\begin{subequations}
	\label{eq:mpc}
	\begin{align}
		V_N(\hat x) = \min_{x,u} \ & \sum_{k=0}^{N-1} \ell(x_k , u_k) + V_\mathrm{f}(x_N) \hspace{-1em} \\
		\mathrm{s.t.} \ & x_0 = \hat x, \\
		&x_{k+1} = f(x_k,u_k), && k\in\mathbb{I}_0^{N-1}, \\
		&h(x_k,u_k) \leq 0, && k\in\mathbb{I}_0^{N-1}, \\
		&h_\mathrm{f}(x_N) \leq 0,
	\end{align}
\end{subequations}
with optimal feedback law $u_N^\star(\hat x)=u_0^\star,$ where we denote the optimal solution as $x^N=(x_0^\star, \ldots,x_N^\star)$, $u^N=(u_0^\star, \ldots,u_{N-1}^\star)$. 
The definition of the EMPC control law entails that, at each time, the state $\hat x$ is measured, Problem~\eqref{eq:mpc} is solved, and the first optimal input $u_0^\star$ is applied to the system. At the next time instant, this procedure is repeated to close the loop.

In this paper, we study conditions under which the EMPC feedback law yields asymptotic stability, first in the infinite-horizon case and afterwards in the finite-horizon case with proper terminal conditions.
In order to set the foundation for our analysis, we define next two OCPs, one forward in time and one backward in time, which will be instrumental in proving our results later on. In a second step, we will also propose a relaxation of those two OCPs, as that will help us deal with some technicalities in the proofs.

We define the following OCP forward in time:
\begin{subequations}
	\label{eq:ocp_fwd}
	\begin{align}
		V_+(\hat x) := \inf_{x,u} \ & \liminf_{N\to\infty}\sum_{k=0}^{N-1} \ell(x_k , u_k) \\
		\mathrm{s.t.} \ & x_0 = \hat x, \\
		&x_{k+1} = f(x_k,u_k),  && k\in\mathbb{I}_0^{\infty}, \label{eq:ocp_fwd_dynamics}\\
		&h(x_k,u_k) \leq 0,  && k\in\mathbb{I}_0^{\infty}, \\
		&\lim_{k\to\infty} x_k = 0, \label{eq:ocp_fwd_tc}
	\end{align}
\end{subequations}
with optimal feedback law $u_+(\hat x)=u_0^+,$ where we denote the optimal solution as $x^+=(x_0^+,x_1^+, \ldots)$, $u^+=(u_0^+,u_1^+, \ldots)$. The feedback law $u_+(\hat x)$ is defined for the first time instant only, but the controls at subsequent times satisfy $u_k^+ \in u_+(x_k)$. Moreover, as the optimal policy might not be unique, $u_+(\hat x)$ is in general a set containing all optimal control values for $\hat x$. In case Problem~\eqref{eq:ocp_fwd} is infeasible we define $V_+(x)=\infty$, and we denote the domain of~\eqref{eq:ocp_fwd} as $\mathcal{X}_+:= \{ \, x \,|\, V_+(x) < \infty \,\}$, which is forward invariant by construction, i.e.,
\begin{align*}
	x&\in\mathcal{X}_+ \quad \implies \quad f(x,u_+(x))\in\mathcal{X}_+.
\end{align*}
Note that Problem~\eqref{eq:ocp_fwd} is essentially Problem~\eqref{eq:mpc} in the limit $N\to\infty$ for a proper definition of $h_\mathrm{f}$ and $V^\mathrm{f}$. 

We further define the following OCP backward in time:
\begin{subequations}
	\label{eq:ocp_bwd}
	\begin{align}
		V_-(\hat x) := \sup_{x,u} \ & 
		\limsup_{N\to\infty}\sum_{k=-N}^{-1}- \ell(x_k , u_k) \hspace{-2em} \\
		\mathrm{s.t.} \ &  x_0 = \hat x, \\
		&  x_0 \in\mathcal{X}_h, \\
		&x_{k+1} = f(x_k,u_k),  && k\in\mathbb{I}_{-\infty}^{-1}, \label{eq:ocp_bwd_dynamics}\\
		&h(x_k,u_k) \leq 0, && k\in\mathbb{I}_{-\infty}^{-1},  \\
		&\lim_{k\to-\infty} x_k = 0,\label{eq:ocp_bwd_tc}
	\end{align}
\end{subequations}
with optimal feedback law $u_-(\hat x)=u_{-1}^-,$ where we denote the optimal solution as $x^-=(x_0^-,x_{-1}^-, \ldots)$, $u^-=(u_{-1}^-,u_{-2}^-, \ldots)$.
In case Problem~\eqref{eq:ocp_bwd} is infeasible we define $V_-(x)=-\infty$, and we denote the domain of~\eqref{eq:ocp_bwd} as $\mathcal{X}_-:=\{ \, x \,|\, V_-(x) > -\infty \,\}$.

\begin{Remark}
	As we will prove that $V_+$ and $V_-$ play an important role for (strict) dissipativity, it is important to relate them to the concepts of \emph{available storage} (defined similarly to $-V_+$) and \emph{required supply} (defined similarly to $-V_-$), introduced already in the seminal papers~\cite{Willems1971,Willems1972a,Willems1972b} for the continuous time case, whose discrete-time formulation can be found in~\cite{Lopezlena2006}. The similarity mentioned above will become very clear in Lemma~\ref{lem:diss_implies_boundedness}. In~\cite{Faulwasser2021} the constrained continuous-time case is discussed, and storage functions are related to Hamilton-Jacobi-Bellman equations and value functions; some known results about dissipativity which we will recall later are derived for the continuous-time case. The main difference with respect to the common definition is the presence of (i) the terminal constraint in the formulation of $V_+$, with a fixed infinite horizon; and (ii) the path constraints enforcing $(x,u)\in\mathcal{H}$. In this light, the fact that $V_+$ and $V_-$ will play an important role in the remainder of this paper does not come as a surprise. Finally, the fact that in the original definition of available storage the prediction horizon is optimized and no terminal constraint is present is related to the fact that the storage function is required to be nonnegative. This is a major difference between ``standard'' dissipativity theory and dissipativity in the context of economic MPC, though on compact sets storage functions $\lambda$ can be shifted by a constant to make them nonnegative, in which case, one might have $\lambda(0)\neq0$.
\end{Remark}

We observe that, in general, 
\begin{align*}
	x&\in\mathcal{X}_+ \quad \centernot\implies \quad x\in\mathcal{X}_-, \\
	x&\in\mathcal{X}_- \quad \centernot\implies \quad x\in\mathcal{X}_+.
\end{align*}
As this would be an issue for some of the results we will prove in this paper, we introduce next the relaxed version of OCP~\eqref{eq:ocp_fwd} and~\eqref{eq:ocp_bwd},
where we relax the system dynamics by introducing the fictitious control variable $z\in\mathbb{R}^{n_x}$:
	\begin{align}
		\label{eq:relaxed_dynamics}
		x_{k+1} &= f(x_k, u_k) + z_k.
	\end{align}
We also introduce a penalty $p$ which we will discuss in further detail later.
The relaxed forward OCP reads
\begin{subequations}
	\label{eq:ocp_fwd_relaxed}
	\begin{align}
		\hspace{-6pt}V_\oplus(\hat x) := \inf_{x,u,z} \ & \liminf_{N\to\infty} \sum_{k=0}^{N-1} \ell(x_k , u_k) + p(x_k,u_k,z_k) \hspace{-6em}&&\hspace{4em} \\
		\mathrm{s.t.} \ & x_0 = \hat x, \\
		&x_{k+1} = f(x_k,u_k) + z_k,  && k\in\mathbb{I}_0^{\infty}, \\
		&h(x_k,u_k) \leq 0, && k\in\mathbb{I}_0^{\infty},  \\
		&\lim_{k\to\infty} x_k = 0, \label{eq:ocp_fwd_relaxed_tc}
	\end{align}
\end{subequations}
with optimal feedback law $\xi_\oplus(\hat x):=(u_\oplus(\hat x), z_\oplus(\hat x))$, where $u_\oplus(\hat x)=u_0^\oplus$, $z_\oplus(\hat x)=z^\oplus_0$, and we denote the optimal solution as $x^\oplus=(x_0^\oplus,x_1^\oplus, \ldots)$, $u^\oplus=(u_0^\oplus,u_{1}^\oplus, \ldots)$, $z^\oplus=(z_0^\oplus,z_{1}^\oplus, \ldots)$. We denote the domain of~\eqref{eq:ocp_fwd_relaxed} as $\mathcal{X}_\oplus:= \{ \, x \,|\, V_\oplus(x) < \infty \,\}$.

The relaxed backward OCP reads
\begin{subequations}
	\label{eq:ocp_bwd_relaxed}
	\begin{align}
		V_\ominus(\hat x) := \sup_{x,u,z} \ & 
		\limsup_{N\to\infty}\sum_{k=-N}^{-1}- \ell(x_k , u_k) - p(x_k,u_k,z_k) \hspace{-8em}&&\hspace{8em} \\
		\mathrm{s.t.} \ &  x_0 = \hat x, \\
		&  x_0 \in\mathcal{X}_h, \\
		&x_{k+1} = f(x_k,u_k) + z_k,  && k\in\mathbb{I}_{-\infty}^{-1},  \\
		&h(x_k,u_k) \leq 0,  && k\in\mathbb{I}_{-\infty}^{-1},  \\
		&\lim_{k\to-\infty} x_k = 0,
	\end{align}
\end{subequations}
with optimal feedback law $\xi_\ominus(\hat x):=(u_\ominus(\hat x), z_\ominus(\hat x))$, where $u_\ominus(\hat x)=u_{-1}^\ominus$, $z_\ominus(\hat x)=z^\ominus_{-1}$, and we denote the optimal solution as $x^\ominus=(x_0^\ominus,x_{-1}^\ominus, \ldots)$, $u^\ominus=(u_{-1}^\ominus,u_{-2}^\ominus, \ldots)$, $z^\ominus=(z_{-1}^\ominus,z_{-2}^\ominus, \ldots)$. 

Note that the feedback law $\xi_\ominus(x)$ is anticausal, as  the state update backward in time from $x$ to $x_-$ when applying control $\xi_\ominus(x)$ is implicitly defined as
\begin{align*}
	x = f(x_-,u_\ominus(x)) + z_\ominus(x),
\end{align*}
which we could formulate explicitly as
\begin{align*}
	x_- = f^{-1}(x-z_\ominus(x),u_\ominus(x)).
\end{align*}
Note that in OCP~\eqref{eq:ocp_bwd_relaxed} $z_\ominus(x)$ is by construction chosen such that $f^{-1}$ is defined for $x-z_\ominus(x)$ and $u_\ominus(x)$. In the remainder of this paper, we will avoid using $f^{-1}$, as proving our results will be easier when using $f$.

In case Problem~\eqref{eq:ocp_bwd_relaxed} is infeasible we define $V_\ominus(x)=-\infty$, and we denote the domain of~\eqref{eq:ocp_bwd_relaxed} as $\mathcal{X}_\ominus:=\{ \, x \,|\, V_\ominus(x) > -\infty \,\}$. Note that, thanks to the relaxation, we have
\begin{align*}
	\mathcal{X}_\ominus = \mathcal{X}_\oplus = \mathcal{X}_h.
\end{align*}
Consequently, as $\mathcal{X}_+\subseteq\mathcal{X}_h$ and $\mathcal{X}_-\subseteq\mathcal{X}_h$, we have
\begin{align*}
	\mathcal{X}_\ominus \supseteq \mathcal{X}_+, &&\mathcal{X}_\oplus \supseteq \mathcal{X}_-.
\end{align*}

Throughout the paper, we will rely on the following assumption. 
\begin{Assumption}
	\label{ass:regularity}
	Functions $\ell$, $f$ and $h$ are twice continuously differentiable, 
	bounded for all bounded $x,u$, and $f(0,0)=0$, $\ell(0,0)=0$,  $h(0,0)\leq0$. For every finite $x\in\mathcal{X}_+$ with $V_+(x)>-\infty$ the infimum of~\eqref{eq:ocp_fwd} is attained and for every finite $x\in\mathcal{X}_-$ with $V_-(x)<\infty$ the supremum of~\eqref{eq:ocp_bwd} is attained. Moreover, $V_+$ and $-V_-$ are lower semicontinuous. While the optimal input might be non-unique, we assume that, for every feasible $\hat x$, the optimal state trajectories of~\eqref{eq:ocp_fwd} and~\eqref{eq:ocp_bwd} are unique and single-valued.
	Finally:
	\begin{enumerate}[label=(\roman*)]
		\item[(i)] The optimal cost-to-go functions $V_+$ and $V_-$ are continuous on, respectively, $\mathcal{B}_\delta\cap \mathcal{X}_+$ and $\mathcal{B}_\delta\cap \mathcal{X}_-$, where $\mathcal{B}_\delta:=\{\,x\,|\, \|x\|\leq \delta \,\}$. 
		Moreover, $\mathcal{B}_\delta\cap \mathcal{X}_h\subseteq\mathcal{X}_+$.
		\item[(ii)] 
		Denote optimal state trajectories of~\eqref{eq:ocp_fwd} and~\eqref{eq:ocp_bwd} as $\{x_k^+\}_{k\geq0}$ and $\{x_k^-\}_{k\leq0}$, respectively.
			For every $\xi>0$, there exist compact sets $\mathcal X_+^\xi\subseteq\mathcal X_h$ and $\mathcal X_-^\xi\subseteq\mathcal X_h$ such that, for all $x_0^+\in\mathcal X_+\cap\mathcal B_\xi$ and $x_0^-\in\mathcal X_-\cap\mathcal B_\xi$, we have $x_k^+\in\mathcal X_+^\xi$, $\forall \, k\geq0$, and $x_k^-\in\mathcal X_-^\xi$, $\forall\, k\leq0$. 
			For every $\xi>0$, there exists a compact set $\mathcal U_+^\xi\subset\mathbb R^{n_u}$ such that, for every $x\in\mathcal X_+\cap\mathcal B_\xi$, there exists $u\in\mathcal U_+^\xi$ satisfying $(x,u)\in\mathcal H$. Set $\mathcal X_h$ is closed.
	\end{enumerate}
\end{Assumption}

Twice continuous differentiability of $f$, $\ell$, and $h$ is a technical regularity assumption which can be at least partially relaxed, at the expense of additional technicalities in the proofs.
The additional assumptions on the system dynamics and stage cost are without loss of generality, provided that the system is optimally operated at steady state, as any steady state can be shifted to the origin. Assuming that the infimum and supremum are attained and that $V_+$, $-V_-$ are lower semicontinuous are regularity assumptions on the OCPs. Uniqueness of the state trajectory ensures that closed-loop and open-loop trajectories match. 
Finally, lower semicontinuity and closedness of the projection of the constraint set are technical assumptions that could maybe be relaxed at the expense of further technicalities. One important implication of this last assumption, which we will use throughout the paper, is that, $V(x)>0 \ \forall \, x\neq0$, $V$ lower semicontinuous $\implies V(x)\geq \gamma(\|x\|)$ with $\gamma\in\mathcal{PD}$, whose definition is provided below.

Assumption~\ref{ass:regularity}~\textit{(i)} is essentially a weak controllability condition involving the cost. 
The additional requirement that $\mathcal{B}_\delta\cap \mathcal{X}_h\subseteq\mathcal{X}_+$ strengthens the requirement for $V_+$, as the origin must be reachable from every feasible state in a neighborhood of the origin. 
Assumption~\ref{ass:regularity}~\textit{(ii)} is a further regularity requirement which automatically holds in case the constraint set $\mathcal{H}$ is compact.

\begin{Definition}[Comparison functions]
	\label{def:k_function}
	Function $\alpha:\mathbb{R}_{\geq0}\mapsto\mathbb{R}_{\geq0}$ is class $\mathcal{K}$ ($\alpha\in\mathcal{K}$) if it is continuous and 
	\begin{align*}
		\alpha(0)=0, && \alpha(y)> \alpha(t), \ \forall \ y>t. 
	\end{align*}
	Function $\alpha:\mathbb{R}_{\geq0}\mapsto\mathbb{R}_{\geq0}$ is class $\mathcal{K}_\infty$ ($\alpha\in\mathcal{K}_\infty$) if 
	\begin{align*}
		\alpha\in\mathcal{K}, && \lim_{y\rightarrow\infty} \alpha(y) = \infty.
	\end{align*}
	Function $\alpha:\mathbb{R}\mapsto\mathbb{R}_{\geq0}$ is positive definite ($\alpha\in\mathcal{PD}$) if it is continuous and 
	\begin{align*}
		\alpha(0)=0, && \alpha(y)>0, \ \forall \ y\neq0.
	\end{align*}
	Function $\sigma:\mathbb{R}_{\geq0}\mapsto\mathbb{R}_{>0}$ is class $\mathcal{L}$ ($\sigma\in\mathcal{L}$), if it is continuous, strictly decreasing, and $\lim_{t\to\infty} \sigma(t) = 0$.
	Function $\beta:\mathbb{R}_{\geq0}\times\mathbb{R}_{\geq0}\mapsto\mathbb{R}_{\geq0}$ is class $\mathcal{KL}$ ($\beta\in\mathcal{KL}$) if 
	\begin{align*}
	\beta(\cdot,t)\in\mathcal{K}, \ \forall\,t\in\mathbb{R}_{\geq0}, \qquad \beta(s,\cdot)\in\mathcal{L}, \ \forall\,s\in\mathbb{R}_{>0}.
	\end{align*}
\end{Definition}

\begin{Proposition}[{\cite[Proposition B.25]{Rawlings2017}}]
	\label{prop:bound}
	Let $\mathcal X\subseteq\mathbb R^{n_x}$ be a set containing the origin, not necessarily closed, and let $V:\mathcal X\to\mathbb R$. Assume that $V$ is continuous at the origin relative to $\mathcal X$, $V(0)=0$, and $V$ is bounded on bounded subsets of $\mathcal X$. Then there exists a class $\mathcal K$ function $\alpha$ such that $V(x)\leq\alpha(\|x\|)$, for all $x\in\mathcal X$.
\end{Proposition}
\begin{pf}
	Define $\check V$ on $\operatorname{cl}(\mathcal X)$, with $\check V(x)=V(x)$ for $x\in\mathcal X$, and $\check V(x)=0$ otherwise. Then $\check V$ is continuous at the origin and bounded on bounded subsets of $\operatorname{cl}(\mathcal X)$. The claim follows by applying~\cite[Proposition B.25]{Rawlings2017} to $\check V$ and restricting the resulting bound to $\mathcal X$. $\hfill\qed$
\end{pf}
This result will allow us to establish the necessary upper bound on the rotated value functions in order to prove that they are Lyapunov functions and, hence, prove asymptotic stability. 

We introduce next the following assumption further relating $V_\oplus$ to $V_+$ and $V_\ominus$ to $V_-$.
\begin{Assumption}
	\label{ass:relaxation}
	Function $p(x,u,z)$ is bounded for all bounded $x,u,z$, and such that $p(x,u,0)=0$, $p(x,u,z)>0$ for all $z\neq0$,  and
	\begin{align*}
		V_\oplus(x)&=V_+(x), \quad \forall \, x\in\mathcal{X}_+, \\
		V_\ominus(x)&=V_-(x), \quad \forall \, x\in\mathcal{X}_-.
	\end{align*}
	Additionally, for every finite $x\in\mathcal{X}_h$ if $V_\oplus(x)>-\infty$ the infimum of~\eqref{eq:ocp_fwd_relaxed} is attained and if $V_\ominus(x)<\infty$ the supremum of~\eqref{eq:ocp_bwd_relaxed} is attained. While the optimal input might be non-unique, we assume that, for every feasible $\hat x$, the optimal state trajectories of~\eqref{eq:ocp_fwd_relaxed} and~\eqref{eq:ocp_bwd_relaxed} are unique and single-valued. Moreover, $V_\oplus(x)$ and $-V_\ominus(x)$ are lower semicontinuous and  continuous on $\mathcal{B}_\delta\cap \mathcal{X}_h$.
	Finally, for every sequence $(x_j,u_j)\in\mathcal{H}$ such that $x_j\to0$, $f(x_j,u_j)\to0$, it holds that $p(x_j,u_j,-f(x_j,u_j))\to0$.
\end{Assumption}
When $\mathcal{X}_+=\mathcal{X}_-$, arguments confined to this common domain can use $V_+$ and $V_-$ directly.
The relaxation provides value functions on the whole set $\mathcal{X}_h$.
Note that, if $V_+$ and $V_-$ are continuous, then on compact sets their variations can be bounded by a suitable modulus of continuity.
Consequently, possibly under additional regularity requirements, it seems possible to select a function $\zeta\in\mathcal{K}_\infty$ and a sufficiently large weighting function $r$ such that a penalty of the form $p(x,u,z) = r(x,u)\zeta(\|z\|)$ dominates the variation of the value functions induced by the relaxation $z$, thereby yielding $V_\oplus(x)=V_+(x)$ on $\mathcal{X}_+$ and $V_\ominus(x)=V_-(x)$ on $\mathcal{X}_-$. If the value functions are furthermore locally Lipschitz, uniformly on the relevant compact sets, then it seems possible to choose $\zeta$ linear, recovering the standard $\ell_1$ exact penalty $p(x,u,z) = r(x,u)\|z\|_1$ for a sufficiently large but finite $r$.
A formal characterization of sufficient conditions for exactness could be obtained using standard results on nonsmooth exact penalties in
nonlinear programming; see, e.g.,~\cite{Nocedal2006,Fletcher1987}. The remaining requirements in Assumption~\ref{ass:relaxation} are
regularity assumptions needed in the subsequent analysis.
The last condition is automatically satisfied if, in a neighborhood of the origin, $p$ is independent of $u$ and continuous in $(x,z)$.

Note that the idea of exploiting an object similar to $V_-$, and especially its relaxed version $V_\ominus$ has been proposed in the context of economic MPC in~\cite{Houska2015} for continuous-time periodic EMPC formulations with return constraints, and in~\cite{Houska2017} in the context of discrete-time EMPC formulations, where the concept of cost-to-travel was introduced. We will comment about the similarities and differences with those approaches in Section~\ref{sec:cost-to-travel}.

As our main interest in this paper is to characterize under which conditions optimizing stage cost $\ell$ yields asymptotic stability, we recall next its definition.
\begin{Definition}[Stability]
	Given a closed-loop system $x_{k+1}=f(x_k,u(x_k))=\bar f(x_k)$ with state trajectory $x_0,\ldots,x_k,\ldots$, the origin is
	\begin{enumerate}
		\item[(i)] locally \emph{stable} in $\mathcal{X}$ if for every $\epsilon>0$ there exists $\delta(\epsilon)>0$ such that $\|x_k\|<\epsilon$ for all $k$ and all $x_0\in\mathcal{X}\cap\{\,x \,|\,\|x\|<\delta(\epsilon) \,\}$; 
		\item[(ii)] locally \emph{asymptotically stable} in $\mathcal{X}$ if it is locally stable in $\mathcal{X}$ and there exists $\phi>0$ such that $\displaystyle\lim_{k\to\infty}x_k=0$ for all $x_0\in\mathcal{X}\cap\{\,x \,|\,\|x\|<\phi \,\}$;
		\item[(iii)] $\mathcal{KL}$-stable in $\mathcal{X}$ if there exists $\beta\in\mathcal{KL}$ such that $\|x_k\|\leq \beta(\|x\|,k), \forall \, x \in\mathcal{X}$.
	\end{enumerate}
\end{Definition}

Because any $n$-norm can be used in the definition above, whenever in this paper we write $\|\cdot\|$ we mean any $n$ norm $\|\cdot\|_n$, $1\leq n\leq\infty.$
A fundamental role in proving asymptotic stability for EMPC is played by dissipativity concepts. We recall next the standard definition of \emph{dissipativity} and \emph{strict dissipativity}, to which we add a new concept which we call \emph{two-storage strict dissipativity}. 

\begin{Definition}[Dissipativity]
	\label{def:dissipativity}
	Consider the inequality
	\begin{align}
		\label{eq:strict_dissipativity}
		&L(x,u) := \ell(x,u) 
		+ \lambda(x) - \lambda(f(x,u))\geq \rho(\|x\|), \\
		&\forall \,(x,u)\in\mathcal H_{\mathcal X}:=\{\,(x,u)\in\mathcal H\,|\,x\in\mathcal X,\ f(x,u)\in\mathcal X\,\}. \nonumber
	\end{align}
	\begin{itemize}
		\item[(i)] \emph{Dissipativity} holds on set $\mathcal{X}$ if there exists a function $\lambda(x)$ bounded on $\mathcal{X}$ and continuous at the origin satisfying $\lambda(0)=0$ and~\eqref{eq:strict_dissipativity} $\forall\,x\in\mathcal{X}$  with $\rho(a)=0, \, \forall\,a$.
		\item[(ii)] \emph{Strict dissipativity} holds on set $\mathcal{X}$ if there exists a function $\lambda(x)$ bounded on $\mathcal{X}$ and continuous at the origin satisfying $\lambda(0)=0$ and~\eqref{eq:strict_dissipativity} $\forall\,x\in\mathcal{X}$ with $\rho \in \mathcal{PD}$.
		\item[(iii)] \emph{Two-storage strict dissipativity} holds on set $\mathcal{X}$ if there exist two functions $\lambda_1(x)$ and $\lambda_2(x)$ bounded on $\mathcal{X}$ and continuous at the origin, satisfying $\lambda_1(0)=\lambda_2(0)=0$ and~\eqref{eq:strict_dissipativity} $\forall\,x\in\mathcal{X}$  with $\rho(a)=0, \,\forall\,a$, and, for $\gamma\in \mathcal{PD}$, 
		\begin{align*}
			\lambda_1(x)\geq\lambda_2(x)+\gamma(\|x\|), \quad \forall \, x \in \mathcal{X}.
		\end{align*}
	\end{itemize}
\end{Definition}
Function $\lambda$ is called \emph{storage function}. 
Some strict dissipativity definitions require $\rho$ to be a $\mathcal{K}$ function. However, if $\mathcal{X}$ is bounded and contains the origin, any positive definite function can be lower-bounded by a $\mathcal{K}$ function~\cite[Lemma~10]{Amrit2011a}. Consequently, that distinction is often not relevant in practice. In some cases the storage functions are required to be nonnegative or continuous. However, continuity at the origin and boundedness are sufficient to obtain comparison-function bounds by Proposition~\ref{prop:bound}. Whenever boundedness on larger bounded sets is required, it is assumed separately. 
In the context of this paper, we are especially interested in $\mathcal{X}=\mathcal{X}_h$. 
Finally, the requirement that storage functions be $0$ at the origin is without loss of generality, as shifting those functions by a constant does not change their properties. However, for two-storage strict dissipativity $\lambda_i(0)=0$ is a structural requirement.

The requirement that $\lambda$ be bounded guarantees that the rotated cost is well defined, by avoiding adding and subtracting $\pm\infty$. The concept of pre-dissipativity, introduced first in~\cite{Gruene2018} drops the usual requirement that $\mathcal{X}$ be compact and relaxes the requirement that $\lambda$ be bounded with the 
requirement that $\lambda$ be bounded on all bounded subsets of $\mathcal{X}_h\subseteq\mathbb{R}^{n_x}$. 

\begin{Remark}
	The common definition of dissipativity requires $\mathcal{X}$ to be compact. In case $\mathcal{X}$ is not bounded but the conditions in Definition~\ref{def:dissipativity} hold for all bounded $x\in\mathcal{X}$, then: (i) \emph{pre-dissipativity}; (ii) \emph{strict pre-dissipativity}; (iii) \emph{two-storage strict pre-dissipativity} hold, respectively. In this paper, with slight abuse of terminology, we will omit the term ``pre'' for the sake of simplicity.
\end{Remark}

The cost $L$ defined in~\eqref{eq:strict_dissipativity} is called \emph{rotated} cost, and is used to prove asymptotic stability by defining a \emph{rotated} OCP, which has the same constraints as the original OCP, but the original stage and terminal cost replaced by the rotated stage and terminal cost, the latter defined as $\bar V_\mathrm{f}(x):=V_\mathrm{f}(x)+\lambda(x)$. The key observation for the proof is that original and rotated OCPs have the same primal solution~\cite[Lemma~14]{Amrit2011a}. 
The rotated value functions become
\begin{align*}
	\bar V_\dagger(x) &= V_\dagger(x) + \lambda(x), &\dagger\in\{+,-\}.
\end{align*}
The proof of this fact can be found in, e.g.,~\cite[Lemma~14]{Amrit2011a} for $V_+$, exploiting the fact that the rotated cost yields a telescopic sum in which only the storage function evaluated at the initial state survives.  The adaptation for $V_-$ is straightforward, following the same approach. 
For completeness and in order to avoid possible confusion, we remark that, by optimality of $V_-$ we have
\begin{align*}
	V_-(f(x,u)) &\geq -\ell(x,u) + V_-(x), \\
	&=-L(x,u) + \lambda(x) - \lambda(f(x,u)) + V_-(x),
\end{align*}
which holds with equality along an optimal transition.
Consequently, $\bar V_-(x)=V_-(x)+ \lambda(x)$ yields
\begin{align*}
	\bar V_-(f(x,u)) &\geq -L(x,u) + \bar V_-(x),
\end{align*}
which holds with equality along an optimal transition.

As we will discuss later, dissipativity is tightly related to the Bellman equation
\begin{align}
	\label{eq:bellman}
	V(x) &= \min_{u:f(x,u)\in\mathcal{X}} \ell(x,u) + H(x,u) + V(f(x,u)), \\
	V(0)&=0. \nonumber
\end{align}
It is therefore fundamental to stress that the Bellman equation is solved by $V_+$ on
$\mathcal X_+$, and by $V_-$ on the subset $\mathcal X_-^-$ introduced in
Lemma~\ref{lem:Vminus_bellman}. 
The condition $V(0)=0$ is added without loss of generality and removes the possibility of shifting the function by a constant, making the equation consistent with~\eqref{eq:ocp_fwd}-\eqref{eq:ocp_bwd} in case dissipativity holds. Constraint $f(x,u)\in\mathcal{X}$ is introduced to make sure that $V(f(x,u))$ is defined for all successor states, and set $\mathcal{X}$ needs to be further specified, e.g., as $\mathcal{X}_+$ or $\mathcal{X}_-$. We further define
\begin{align*}
	H(x,u) := \left \{ \begin{array}{ll}
		0 & \text{if } h(x,u) \leq 0,\\
		\infty & \text{otherwise}
	\end{array} \right ..
\end{align*}
The Bellman equation entails that for every $(x,u)\in\mathcal{H}$, with $f(x,u)\in\operatorname{dom}V$,  $|V(x)|<\infty$ we have
\begin{align}
	\label{eq:bellman_inequality}
	V(x) &\leq \ell(x,u) + V(f(x,u)), && V(0)=0.
\end{align}
While that is clear for $V_+$, it is a bit less obvious for $V_-$, such that we state the following lemma.
	\begin{Lemma}
		\label{lem:Vminus_bellman}
		Function $V_-$ satisfies~\eqref{eq:bellman_inequality} for all $f(x,u)\in\mathcal{X}_-$, $(x,u)\in\mathcal{H}$, and solves the Bellman equation~\eqref{eq:bellman} for all $x\in\mathcal{X}_-^-$, where
		\begin{align*}
			\mathcal{X}_-^- := \{\,x_-(x)  \,|\,   x\in\mathcal X_- \,\},
		\end{align*}
		where $x_-(x)=x_{-1}^-$ denotes the first state backward in time in an optimal solution of OCP~\eqref{eq:ocp_bwd} with terminal state $x_0=x$.
	\end{Lemma}
	\begin{pf}
		For $V_-$, the Bellman equation reads
		\begin{align*}
			V_-(x) = \max_{u,x_-} \ & -\ell(x_-,u) - H(x_-,u) + V_-(x_-) \\
			\mathrm{s.t.} \ & x=f(x_-,u).
		\end{align*}
		Then, for all $x_-$ such that $x=f(x_-,u)\in\mathcal{X}_-$
		\begin{align*}
			V_-(x) \geq  -\ell(x_-,u) - H(x_-,u) + V_-(x_-).
		\end{align*}
		Consequently, for all $f(x,u)\in\mathcal{X}_-$, we have
		\begin{align*}
			V_-(f(x,u)) \geq -\ell(x,u) - H(x,u) + V_-(x),
		\end{align*}
		which also reads
		\begin{align}
			\label{eq:bellman_inequality_Vminus}
			V_-(x) \leq V_-(f(x,u)) + \ell(x,u) + H(x,u).
		\end{align}
		Moreover, we know that for all $x\in\mathcal{X}_-$ there exists $u_-(x)$ and  $x_-(x)$ such that $x=f(x_-(x),u_-(x))$ and
		\begin{align*}
			V_-(x) = -\ell(x_-(x),u_-(x)) - H(x_-(x),u_-(x)) + V_-(x_-(x)),
		\end{align*}
		i.e., such that~\eqref{eq:bellman_inequality_Vminus} holds with equality. Hence, $V_-$ solves the Bellman equation~\eqref{eq:bellman} for all $x\in\mathcal{X}_-^-$. $\hfill\qed$
	\end{pf}
	\begin{Remark}
		Function $V_-$ satisfies the Bellman inequality \eqref{eq:bellman_inequality_Vminus} for admissible transitions within $\mathcal X_-$, and the causal Bellman equality \eqref{eq:bellman} on $\mathcal X_-^-$, with $\mathcal X=\mathcal X_-$ in \eqref{eq:bellman}. In the unconstrained linear-quadratic case, existence of a finite quadratic representation of $V_-$ must be distinguished from its representing matrix also solving the causal CGDARE. If $\mathcal X_-^-=\mathbb R^{n_x}$, then $V_-$ satisfies the causal Bellman equation globally. Otherwise, Lemma~\ref{lem:Vminus_bellman} only guarantees this equality on $\mathcal X_-^-$, though this does not preclude the existence of a finite quadratic backward value on $\mathbb R^{n_x}$. More details on these aspects in relation to Riccati equations can be found in~\cite{Ionescu1996,Zanon2025a} and references therein.
	\end{Remark}

As proven in the lemma above, for every $(x,u)\in\mathcal{H}$, with $f(x,u)\in\operatorname{dom}V_-$, $|V_-(x)|<\infty$, optimality entails 
\begin{align}
	\label{eq:V_minus_optimality}
	V_-(f(x,u)) \geq -\ell(x,u) + V_-(x).
\end{align}
Similarly, also $V_\oplus$ and $V_\ominus$ solve a Bellman equation, though with their corresponding cost $\ell(x,u) + p(x,u,z)$ 
and dynamics $x_+=f(x,u)+z$.
In particular, we observe that, for every $(x,u)\in\mathcal{H}$ with $|V_\oplus(x)|<\infty$, $|V_\ominus(x)|<\infty$ and $f(x,u)+z\in\mathcal{X}_h$ we have 
\begin{align}
	V_\oplus(x) \leq \ell(x,u) + p(x,u,z)
	 + V_\oplus(f(x,u)+z), \label{eq:V_plus_optimality_relaxed_inequality_z}\\
	V_\ominus(f(x,u)+z) \geq -\ell(x,u) - p(x,u,z)
	 + V_\ominus(x).	\label{eq:V_minus_optimality_relaxed_inequality_z}
\end{align}
However, even in the presence of the relaxation of the system dynamics and modified stage cost, for all $x,u$, one can replace $z=0$ in~\eqref{eq:V_plus_optimality_relaxed_inequality_z}-\eqref{eq:V_minus_optimality_relaxed_inequality_z} to obtain that, for every $(x,u)\in\mathcal{H}$, with $|V_\oplus(x)|<\infty$, $|V_\ominus(x)|<\infty$ and $f(x,u)\in\mathcal{X}_h$ we have
\begin{align}
	V_\oplus(x) \leq \ell(x,u) + V_\oplus(f(x,u)), \label{eq:V_plus_optimality_relaxed_inequality}\\
	V_\ominus(f(x,u)) \geq -\ell(x,u) + V_\ominus(x),	\label{eq:V_minus_optimality_relaxed_inequality}
\end{align}
such that also $V_\oplus$, $V_\ominus$ satisfy the Bellman inequality~\eqref{eq:bellman_inequality}.
This property will turn out to be essential to prove some of our results, in particular, $-V_\oplus$, $-V_\ominus$ will be used as storage functions to enforce dissipativity and two-storage strict dissipativity.

The following trivial lemma establishes that any value function is a storage function satisfying dissipativity.
\begin{Lemma}
	\label{lem:value_functions_storage_functions}
	Any value function solving the Bellman equation~\eqref{eq:bellman}, or the Bellman inequality~\eqref{eq:bellman_inequality}, which is continuous at the origin is the negative of a storage function satisfying dissipativity on the set where it is bounded.
\end{Lemma}
\begin{pf}
	The proof follows from~\eqref{eq:bellman_inequality}. $\hfill\qed$
\end{pf}

As we prove next, a function satisfying the causal Bellman equation~\eqref{eq:bellman} cannot yield strict dissipativity on its Bellman domain through the storage function $\lambda=-V$, provided that the minimum is attained.
\begin{Lemma}
	Provided that the minimum is attained, no function $V(x)$ solving the Bellman equation~\eqref{eq:bellman} can be used as a storage function $\lambda(x)=-V(x)$ satisfying strict dissipativity.
\end{Lemma}
\begin{pf}
	While as proven in Lemma~\ref{lem:value_functions_storage_functions} $\lambda(x)=-V(x)$ satisfies dissipativity, for all $\tilde x$ there exists at least one $\tilde u$ such that
	\begin{align*}
		0 = \ell(\tilde x,\tilde u) - V(\tilde x) + V(f(\tilde x,\tilde u))=L(\tilde x,\tilde u). \\[-2em]
	\end{align*}
	$\hfill\qed$
\end{pf}

Although $-V_+$ does not yield strict dissipativity on $\mathcal X_+$, the forward and backward value functions can be combined to obtain two-storage strict dissipativity.
Indeed, as we will prove next, two-storage strict dissipativity implies that the feedback law $u_+$ from~\eqref{eq:ocp_fwd} is asymptotically stabilizing. Conversely, if $u_+$ is asymptotically stabilizing and $V_+$ is bounded on bounded subsets of $\mathcal X_+$, then two-storage strict dissipativity holds with storage functions $\lambda_1=-V_\ominus$ and $\lambda_2=-V_+$.

\section{The Infinite-Horizon Case}
\label{sec:infinite_horizon}

In this section, we discuss sufficiency and necessity of two-storage strict dissipativity for boundedness of the value functions and asymptotic stability.

\subsection{Dissipativity and Boundedness}

We first prove a very useful property of storage functions in relation to the value functions introduced above.
\begin{Lemma}
	\label{lem:diss_implies_boundedness}
	Suppose that Assumption~\ref{ass:regularity} holds. Assume that dissipativity holds on $\mathcal{X}_h$ for some storage function $\lambda(x)$. Then, for all bounded $x\in\mathcal{X}_h$
	\begin{align}
		-V_-(x) \geq \lambda(x) &\geq -V_+(x),  \label{eq:lambda_bound} \\
		V_+(x) - V_-(x) &\geq 0, \label{eq:V+-bound} \\
		V_+(x) &> -\infty, & V_-(x)< \infty. \label{eq:V+-bound_from_below}
	\end{align}
	If strict dissipativity holds on $\mathcal{X}_h$, for all bounded $x\in\mathcal{X}_h$
	\begin{align}
		\label{eq:V+-bound_strong}
		V_+(x) - V_-(x)&\geq \rho(\|x\|),
	\end{align}
	with $\rho\in\mathcal{PD}$ from Definition~\ref{def:dissipativity}(ii).
	If two-storage strict dissipativity holds on $\mathcal{X}_h$, for all bounded $x\in\mathcal{X}_h$
	\begin{align}
		\label{eq:V+-bound_strong2}
		V_+(x) - V_-(x)\geq \gamma(\|x\|),
	\end{align}
	with $\gamma\in\mathcal{PD}$ from Definition~\ref{def:dissipativity}(iii).
\end{Lemma}
\begin{pf}
	By dissipativity there exists a bounded function $\lambda(x)$ such that $L(x,u)\geq0$. Consequently, 
	\begin{align*}
		V_+(x) + \lambda(x)=
		\bar V_+(x) &\geq 0, \\
		V_-(x) + \lambda(x) =	\bar V_-(x)&\leq 0.
	\end{align*}
	Note that, given the definitions of $V_+$ and $V_-$, we have $x\notin\mathcal{X}_+\implies V_+(x)=\infty$, $x\notin\mathcal{X}_-\implies V_-(x)=-\infty$, such that the inequalities remain valid outside $\mathcal{X}_+$ and $\mathcal{X}_-$. 
	The second claim is then directly obtained by noting that the storage function $\lambda(x)$ rotating the two value functions cancels out when taking the difference $\bar V_+(x)-\bar V_-(x)\geq0$. The inequalities in~\eqref{eq:V+-bound_from_below} follow from $\bar V_+(x) \geq 0$, $\bar V_-(x) \leq 0$ and the fact that $\lambda(x)$ is bounded for all bounded $x\in\mathcal{X}_h$.
	
	The proof for the case in which strict dissipativity holds is obtained in the same manner, by observing that in that case $L(x,u)\geq\rho(\|x\|)$, such that
	\begin{align}
		\label{eq:lambda_bound_strict}
		\bar V_+(x) &\geq \rho(\|x\|) \geq 0 \geq \bar V_-(x).
	\end{align}
	The proof for two-storage strict dissipativity is obtained as follows. Define
	\begin{subequations}
		\label{eq:rotations_12}
		\begin{align}
			L_1(x,u) &:= \ell(x,u) + \lambda_{1}(x) - \lambda_{1}(f(x,u)) \geq 0, \\
			L_2(x,u) &:= \ell(x,u) + \lambda_{2}(x) - \lambda_{2}(f(x,u)) \geq 0, \\
			L_1(x,u) &= L_2(x,u) + \lambda_{1,2}(x) - \lambda_{1,2}(f(x,u)),
		\end{align}
	\end{subequations}
	where the cost rotation between $L_1(x,u)$ and $L_2(x,u)$ is done using  $\lambda_{1,2}(x):=\lambda_1(x)-\lambda_2(x)\geq\gamma(\|x\|)$. Each rotated stage cost $L^i$ used in~\eqref{eq:ocp_fwd} and~\eqref{eq:ocp_bwd} yields a corresponding value function $V^i_+$ and $V^i_-$.
	Because
	\begin{align*}
		\bar V_+^1(x)&\geq 0, & \bar V_-^1(x)&\leq 0, \\
		\bar V_+^2(x)&\geq 0, & \bar V_-^2(x)&\leq 0,
	\end{align*}
	we then have
	\begin{subequations}
		\label{eq:rotations_V_12}
		\begin{align}
			\bar V_+^1(x)&=\bar V_+^2(x)+\lambda_{1,2}(x)\geq\gamma(\|x\|), \\
			\bar V_-^2(x)&=\bar V_-^1(x)-\lambda_{1,2}(x)\leq-\gamma(\|x\|).
		\end{align}
	\end{subequations}
	Consequently,
	\begin{subequations}
		\label{eq:lambda_bound_strict2}
		\begin{align}
			\bar V_+^1(x)&\geq\gamma(\|x\|)\geq0\geq \bar V_-^1(x), \label{eq:lambda_bound_strict2_1}\\
			\bar V_+^2(x)&\geq0\geq-\gamma(\|x\|)\geq \bar V_-^2(x),\label{eq:lambda_bound_strict2_2}
		\end{align}
	\end{subequations}
	such that the claim can be proven using either of the two sets of inequalities above.
	$\hfill\qed$
\end{pf}

We assumed for simplicity that dissipativity holds on $\mathcal{X}_h$. As we will prove in Lemma~\ref{lem:boundedness_implies_diss} this is not a restriction with respect to assuming that dissipativity holds only on $\mathcal{X}_+$ or $\mathcal{X}_-$. 

\begin{Remark}
	\label{rem:Vbounds}
	The first result of the lemma, i.e.,~\eqref{eq:lambda_bound} is close to the known result in the seminal paper~\cite{Willems1972a}, though formulated in continuous time with the further restriction that $\lambda(x)\geq0$. A discrete-time counterpart can be found in~\cite{Lopezlena2006} for the unconstrained case. As we discussed earlier, $V_+$ and $V_-$ are respectively related to the so-called \emph{available storage} and \emph{required supply}, though both are typically formulated in an unconstrained setting, i.e., with $\mathcal{H}=\mathbb{R}^{n_x+n_u}$, and $V_+$ has a fixed infinite horizon and a terminal constraint. The unconstrained value functions $V_+^\mathrm{u}$, $V_-^\mathrm{u}$ solve an unconstrained Bellman equation, i.e., with $\mathcal{H}=\mathbb{R}^{n_x+n_u}$. Consequently, the bounds in~\eqref{eq:lambda_bound} can be similarly derived using $V_+^\mathrm{u},V_-^\mathrm{u}$. Because optimality entails $V_+(x)\geq V_+^\mathrm{u}(x)$ and $V_-(x)\leq V_-^\mathrm{u}(x)$, this might seem to contradict the result obtained with the constrained value functions. 
	However, in case $V_+(x)\neq V_+^\mathrm{u}(x)$, then the unconstrained trajectory must violate the constraints at some time, i.e., it must pass through some state-input pair for which the cost rotated using $V_+$ is not defined. Hence, the presence of the constraints allows us to exclude some state-input pairs from the set in which dissipativity must hold and, therefore enlarge the set of storage functions satisfying dissipativity on the smaller set of states $\mathcal{X}_+$. However, if dissipativity must hold on a larger set, then the constraints must be relaxed accordingly and the value function, hence the lower bound on $\lambda$, might increase on some subset of $\mathcal{X}_+$. 
	Note also that, thinking of the constraints as additional stage cost $H$, tighter constraints allow for a larger set of storage functions because they introduce additional (infinite) cost for state-input pairs for which storage function $-V_+$ would otherwise fail to satisfy $L(x,u)\geq0$. We provide an example showing these aspects in Section~\ref{sec:lq_example}. The same observations apply to $V_-$, \emph{mutatis mutandis}.
\end{Remark}

We ought to stress that~\eqref{eq:lambda_bound} has deeper implications than the bound on storage functions satisfying dissipativity. Indeed, by Lemma~\ref{lem:value_functions_storage_functions} this entails that 
$V_\oplus$, $V_\ominus$ are bounded from above and below, respectively by $V_+$ and $V_-$, which is also a direct consequence of the relaxation of the system dynamics. 

Note that, if OCP~\eqref{eq:ocp_bwd}  is infeasible for some $\tilde x$, then $V_-(\tilde x) = -\infty$ and the storage function has no upper bound for $\tilde x$. This is relevant for the case of systems that are not fully controllable, as also discussed in~\cite{Zanon2025,Zanon2025a} for the linear quadratic case.
Moreover, storage functions satisfying strict dissipativity must yield rotated value functions which satisfy the stronger bound~\eqref{eq:lambda_bound_strict}, or the corresponding version for two-storage strict dissipativity~\eqref{eq:lambda_bound_strict2}. This yields the following corollary.

\begin{Corollary}
	\label{cor:V+-strict_diss}
	Assume that $\infty>V_+(\tilde x)=V_-(\tilde x)>-\infty$ for some $\tilde x\neq0$. Then, both strict dissipativity and two-storage strict dissipativity cannot hold.
\end{Corollary}
\begin{pf}
	If $V_+(\tilde x)=V_-(\tilde x)$, then Equation~\eqref{eq:lambda_bound_strict} 
	cannot hold, since $\bar V_+(\tilde x)=V_+(\tilde x)+\lambda(\tilde x)$, $\bar V_-(\tilde x)=V_-(\tilde x)+\lambda(\tilde x)$, and $\rho(\|\tilde x\|)>0.$
	Similarly, Equation~\eqref{eq:lambda_bound_strict2} 
	cannot hold, since $\bar V_+^i(\tilde x)=V_+(\tilde x)+\lambda_i(\tilde x)$, $\bar V_-^i(\tilde x)=V_-(\tilde x)+\lambda_i(\tilde x)$, and $\gamma(\|\tilde x\|)>0.$ Since both strict dissipativity and two-storage strict dissipativity require that~\eqref{eq:strict_dissipativity} holds for all $(x,u)\in\mathcal{H}$, this proves the claim. $\hfill\qed$
\end{pf}
An observation similar to the claim of Corollary~\ref{cor:V+-strict_diss} was discussed in the seminal paper~\cite[Theorem~8]{Willems1971} for the linear-quadratic case in continuous time, although under a different name than strict dissipativity, and by only discussing the implications for asymptotic stability.

Let us comment further on the bounds in Equation~\eqref{eq:lambda_bound}. Assume by contradiction that there exists a storage function satisfying dissipativity and such that for some $\tilde x$ it satisfies $\lambda(\tilde x)<-V_+(\tilde x)$. Then, this would entail that $\bar V_+(\tilde x)=V_+(\tilde x)+\lambda(\tilde x)<0$, which is impossible since $L(x,u)\geq0$. 
The same holds for $V_-$, though a priori not necessarily for $V_\oplus$, $V_\ominus$, as in that case the cost rotation involves $z$, and the rotated cost is now a function of $z$. Hence, $L(x,u)\geq0$ does not necessarily imply $\mathcal{L}(x,u,z)\geq0$, but only $\mathcal{L}(x,u,0)\geq0$, where we define
\begin{align*}
	\mathcal{L}(x,u,z) := \ell(x,u) + p(x,u,z)
	 + \lambda(x) - \lambda(f(x,u)+z),
\end{align*}
Hence, one cannot directly conclude $\bar V_\oplus(x)\geq0$ or $\bar V_\ominus(x)\leq0$, as one can a priori not exclude that $\mathcal{L}(x,u,z)<0$ for some $z\neq0$. In turn, it cannot be said if $-V_\ominus(x)\geq \lambda(x)\geq -V_\oplus(x)$, though the two bounds hold for $x\in\mathcal{X}_-$ and $x\in\mathcal{X}_+$, respectively, if Assumption~\ref{ass:relaxation} holds, as that entails $0\in z_k^\oplus$, $0\in z_k^\ominus$ for all $k$ and $V_+(x)=V_\oplus(x)$, $V_-(x)=V_\ominus(x)$. 

In the next lemma, we prove that assuming that dissipativity holds on $\mathcal{X}_h$ is not restrictive.

\begin{Lemma}
	\label{lem:boundedness_implies_diss}
	Suppose that Assumptions~\ref{ass:regularity} and~\ref{ass:relaxation} hold.
	Assume that either $V_+(x)>-\infty$ for all bounded $x\in\mathcal{X}_+$ or $V_-(x)<\infty$ for all bounded $x\in\mathcal{X}_-$. Then dissipativity holds for  some storage function $\lambda(x)$ for all bounded $x\in\mathcal{X}_h$. Additionally,  for all bounded $x\in\mathcal{X}_h$, 
	$V_\oplus(x)>-\infty$, $V_\ominus(x)<\infty$ and $V_\oplus(x)\geq V_\ominus(x)$.
	Finally, $V_+(x)>-\infty$ for all bounded $x\in\mathcal{X}_+$ if and only if $V_-(x)<\infty$ for all bounded $x\in\mathcal{X}_-$.
\end{Lemma}
\begin{pf}
	For $V_\oplus$ a positive cost $p(x,u,z)>0$ corresponds to any $z\neq 0$. Consequently, as $\ell$ and $f$ are bounded for all bounded $x,u$, it is only possible to have $V_\oplus(\tilde x)=-\infty$ for some bounded $\tilde x$ if there exists some bounded $\hat x$ such that $V_+(\hat x)=-\infty$. To see that, assume by contradiction that $V_\oplus(\tilde x)=-\infty$ for some bounded $\tilde x$. Then, from any $(\hat x,\hat u)\in\mathcal{H}$ select $\hat z = \tilde x - f(\hat x, \hat u)$ to obtain a feasible trajectory with unbounded negative cost, such that $V_\oplus(x)=-\infty$ for all bounded $x\in\mathcal{X}_h$. Because by assumption $V_\oplus(x)=V_+(x)$ for all bounded $x\in\mathcal{X}_+$, this yields the desired result.
	Hence, if $V_+(x)>-\infty$ for all bounded $x\in\mathcal{X}_+$, then $V_\oplus(x)>-\infty$ for all bounded $x\in\mathcal{X}_h$.
	The same proof applies, \emph{mutatis mutandis}, to $V_-$ and $V_\ominus$, such that if $V_-(x)<\infty$ for all bounded $x\in\mathcal{X}_-$, then $V_\ominus(x)<\infty$ for all bounded $x\in\mathcal{X}_h$.
	
	We consider first the case $V_+(x)>-\infty$ for all bounded $x\in\mathcal{X}_+$. For any backward trajectory $\{x_k,u_k,z_k\}_{k\leq-1}$, feasible for~\eqref{eq:ocp_bwd_relaxed} with $x_0=x$, by~\eqref{eq:V_plus_optimality_relaxed_inequality_z}, for every $N\geq1$ we have
	\begin{align*}
		V_\oplus(x_{-N})-\sum_{k=-N}^{-1}
		\left(
		\ell(x_k,u_k)+p(x_k,u_k,z_k)
		\right)
		\leq V_\oplus(x).
	\end{align*}
	Since $\lim_{N\to\infty}x_{-N}=0$ and $V_\oplus$ is continuous at the origin, taking the limit yields
	\begin{align*}
		\limsup_{N\to\infty}
		-\sum_{k=-N}^{-1}
		\left(
		\ell(x_k,u_k)+p(x_k,u_k,z_k)
		\right)
		\leq V_\oplus(x).
	\end{align*}
	Since this holds for every feasible backward trajectory, taking the supremum over all such trajectories yields
	\begin{align*}
		V_\ominus(x)\leq V_\oplus(x)<\infty.
	\end{align*}
	
	We prove next that $V_\ominus(x)$ is bounded on any bounded subset $\mathcal{\bar X}$ of $\mathcal{X}_h$. Let $R<\infty$ be such that $\mathcal{\bar X}\subseteq\mathcal{X}_h\cap\mathcal{B}_R$. Since $\mathcal{X}_h$ is closed, $\mathcal{X}_h\cap\mathcal{B}_R$ is compact. For every $x\in\mathcal{X}_h\cap\mathcal{B}_R$, consider the feasible backward trajectory satisfying $x_k=0$, $u_k=0$, $z_k=0$ for all $k\leq-2$, and $x_{-1}=0$, $u_{-1}=0$, $z_{-1}=x$. Since $f(0,0)=0$, $\ell(0,0)=0$, and $p(0,0,0)=0$, this yields
	\begin{align*}
		V_\ominus(x)\geq-p(0,0,x)>-\infty,
	\end{align*}
	since $p$ is bounded for bounded arguments, such that $V_\ominus$ is bounded from below on $\mathcal{X}_h\cap\mathcal{B}_R$. Moreover, $V_\ominus$ is upper semicontinuous and finite-valued, such that it is bounded from above on $\mathcal{X}_h\cap\mathcal{B}_R$. 
	Consequently, $\lambda(x)=-V_\ominus(x)$ is bounded on any bounded subset of $\mathcal{X}_h$.
	
	Moreover, for all $(x,u)\in\mathcal H$ with $f(x,u)\in\mathcal X_h$, \eqref{eq:V_minus_optimality_relaxed_inequality} yields
	$\ell(x,u)-V_\ominus(x)+V_\ominus(f(x,u))\geq0$. 
	Since $V_\ominus$ is continuous at the origin and $V_\ominus(0)=0$, $\lambda(x)=-V_\ominus(x)$ is a storage function satisfying dissipativity on any bounded subset $\mathcal{\bar X}$ of $\mathcal{X}_h$.
	
	Consider now the case $V_-(x)<\infty$ for all bounded $x\in\mathcal{X}_-$. By the arguments at the beginning of the proof, $V_\ominus(x)<\infty$ for all bounded $x\in\mathcal{X}_h$. Moreover, the feasible backward trajectory constructed above yields $V_\ominus(x)>-\infty$. Hence,	$V_\ominus$ is finite-valued on bounded subsets of	$\mathcal{X}_h$, and the same boundedness and optimality arguments above show that $\lambda(x)=-V_\ominus(x)$ is a storage function satisfying dissipativity on any bounded subset of	$\mathcal{X}_h$. By Lemma~\ref{lem:diss_implies_boundedness}, $V_+(x)>-\infty$ for all bounded $x\in\mathcal{X}_+$. Consequently, the first part of the proof applies and yields $V_\oplus(x)>-\infty$, and $V_\oplus(x)\geq V_\ominus(x)$
	for all bounded $x\in\mathcal{X}_h$.
	
	Hence, boundedness of either $V_+$ or $V_-$ entails that dissipativity holds on any bounded subset of $\mathcal{X}_h$, and that
	\begin{align*}
		V_\oplus(x)>-\infty,\qquad
		V_\ominus(x)<\infty,\qquad
		V_\oplus(x)\geq V_\ominus(x)
	\end{align*}
	for all bounded $x\in\mathcal{X}_h$.
	Finally, by Lemma~\ref{lem:diss_implies_boundedness} and~\eqref{eq:lambda_bound}, $V_+(x)>-\infty$ for all bounded $x\in\mathcal{X}_+$ if and only if $V_-(x)<\infty$ for all bounded $x\in\mathcal{X}_-$, which concludes the proof. $\hfill\qed$
\end{pf}
The fact that finiteness of $V_+$ or $V_-$ rules out, respectively, $V_\oplus(x)=-\infty$ and $V_\ominus(x)=\infty$ for bounded $x\in\mathcal{X}_h$ is important, since by construction $V_\oplus(x)\leq V_+(x)$ and $V_\ominus(x)\geq V_-(x)$. 
Moreover, the proof shows that $-V_\ominus$ is bounded on bounded subsets of $\mathcal{X}_h$ and can therefore be used as a storage function. The corresponding boundedness property of $-V_\oplus$ does not follow directly from those arguments and will be addressed separately in Lemma~\ref{lem:2storage_strict_diss_Xh}.

We proved above that, if dissipativity holds on all bounded subsets of either $\mathcal{X}_+$ or $\mathcal{X}_-$, then it must also hold on any bounded subset of $\mathcal{X}_h$. We will prove later, in Lemma~\ref{lem:2storage_strict_diss_Xh}, that an analogous extension holds for two-storage strict dissipativity under an additional boundedness condition on $V_\oplus$, so that $-V_\oplus$ can then also be used as a storage function.

\subsection{Asymptotic Stability}

So far, we proved sufficiency and necessity of dissipativity for boundedness of $V_+$ and $V_-$. Next, we turn our attention to asymptotic stability. Sufficiency of strict dissipativity for asymptotic stability has been proven in~\cite{Amrit2011a}. 
We will prove next that two-storage strict dissipativity yields asymptotic stability and it entails that strict dissipativity holds along optimal trajectories. Afterwards, we will discuss necessity of two-storage strict dissipativity for asymptotic stability, and we will prove that strict dissipativity implies two-storage strict dissipativity. Note that we need Assumption~\ref{ass:relaxation} only to guarantee that $V_+$ and $V_-$ are bounded on bounded subsets of $\mathcal{X}_+$ and $\mathcal{X}_-$, respectively.

\begin{Theorem}
	\label{thm:2str_diss_implies_str_diss}
	Suppose that Assumptions~\ref{ass:regularity} and~\ref{ass:relaxation} hold. 
	Assume that two-storage strict dissipativity holds for all bounded $x\in \mathcal{X}_+\cup\mathcal{X}_-$.
	Then the system in closed-loop with feedback control law $u_+(x)$ is asymptotically stable forward in time; while the system in closed-loop with feedback control law $u_-(x)$ is asymptotically stable backward in time. If, moreover, the systems in closed loop with $u_+$ forward in time and with $u_-$ backward in time are $\mathcal{KL}$-stable, then strict dissipativity holds along optimal trajectories of~\eqref{eq:ocp_fwd} and~\eqref{eq:ocp_bwd}.
\end{Theorem}
\begin{pf}
	By Lemma~\ref{lem:boundedness_implies_diss} we know that $V_\ominus$ is bounded on bounded subsets of $\mathcal{X}_h$, such that $V_-$ is also bounded on bounded subsets of $\mathcal{X}_-$. By the arguments used in the proof of Lemma~\ref{lem:boundedness_implies_diss}, we can as well prove that $V_+$ is bounded on bounded subsets of $\mathcal{X}_+$. Indeed, by Assumption~\ref{ass:regularity} we know that  for every $x\in\mathcal X_+\cap\mathcal B_\xi$, there exists $u\in\mathcal U_+^\xi$ satisfying $(x,u)\in\mathcal H$, while $f(x,u)$ is bounded because $x$ and $u$ are bounded. Hence, we can construct a trajectory for the relaxed dynamics that goes immediately to the origin and then stays there with zero cost. Since this trajectory is feasible for the relaxed OCP~\eqref{eq:ocp_fwd_relaxed}, for all $x\in\mathcal{X}_+\cap\mathcal B_\xi$ we have that
	\begin{align*}
		V_+(x)=V_\oplus(x) \leq \ell(x,u) + p(x,u,-f(x,u)).
	\end{align*}
	Since $x$ and $u$ range over bounded sets, Assumptions~\ref{ass:regularity} and~\ref{ass:relaxation} imply that the right-hand side is bounded. Hence, $V_+$ is bounded from above on bounded subsets of $\mathcal{X}_+$, and the lower bound follows from Lemma~\ref{lem:diss_implies_boundedness}.
	
	From~\eqref{eq:rotations_12}-\eqref{eq:rotations_V_12}, by rotating using $\lambda_1$ we have
	\begin{align*}
		\beta_1(\|x\|)\geq\bar V_+^1(x) 
		\geq \gamma(\|x\|),
	\end{align*}
	with $\beta_1\in\mathcal{K}$, where the upper bound follows from 
	Proposition~\ref{prop:bound} applied on $\mathcal X_+$. Moreover, because $(x,u_+(x))\in\mathcal{H}$, we have
	\begin{align}
		\label{eq:lyap_decrease0}
		\bar V_+^1(f(x,u_+(x))) - \bar V_+^1(x) = -L_1(x,u_+(x)) \leq 0,
	\end{align}
	such that $u_+(x)$ yields a stable closed-loop system~\cite{Kellet2023}, though this does not yet prove asymptotic stability. However, Constraint~\eqref{eq:ocp_fwd_tc} enforces the missing attractivity condition. Consequently, $u_+(x)$ yields an asymptotically stable closed-loop system.
	
	If the closed-loop system is additionally $\mathcal{KL}$-stable, then by~\cite[Theorem~6.2]{Gruene2014a} there exists a Lyapunov function $\mathcal{V}_+$ satisfying
	\begin{align*}
		\alpha_1(\|x\|) \geq \mathcal{V}_+(x) &\geq \alpha_2(\|x\|), \\
		\mathcal{V}_+(f(x,u_+(x))) - \mathcal{V}_+(x) &\leq -\alpha_3(\|x\|),
	\end{align*}
	with $\alpha_i\in\mathcal{K}$, $i=1,2,3$.  
	Let us rotate $\ell(x,u)$ using $\lambda_3(x) := \lambda_1(x)+\mathcal{V}_+(x)$ to get
	\begin{align}
		\label{eq:L3}
		L_3(x,u) &:= L_1(x,u) + \mathcal{V}_+(x) - \mathcal{V}_+(f(x,u)).
	\end{align}
	While the fact that in general 
	\begin{align*}
		\mathcal{V}_+(x) - \mathcal{V}_+(f(x,u)) \ngeq0
	\end{align*}
	does not allow us to prove strict dissipativity, we can prove that strict dissipativity holds along optimal trajectories of~\eqref{eq:ocp_fwd}, as the Lyapunov decrease condition yields 
	\begin{align*}
		L_3(x,u_+(x)) &\geq L_1(x,u_+(x)) + \alpha_3(\|x\|) \geq \alpha_3(\|x\|).
	\end{align*}
	Consequently, because $\lambda_3(x)$ is continuous at $x=0$, we have that $\bar V_+^3$ is a Lyapunov function because
	\begin{align*}
		\beta_3(\|x\|) \geq \bar V_+^3(x) &\geq \gamma(\|x\|), \\
		\bar V_+^3(f(x,u_+(x))) - \bar V_+^3(x) &= -L_3(x,u_+(x)) \leq -\alpha_3(\|x\|),
	\end{align*}
	such that strict dissipativity holds along optimal trajectories, and $u_+(x)$ is asymptotically stabilizing (forward in time). The upper bound follows from Proposition~\ref{prop:bound} applied on $\mathcal X_+$.
	
	We now turn to the stability properties of $u_-$. We observe that, from~\eqref{eq:rotations_12}-\eqref{eq:rotations_V_12},  by rotating using $\lambda_2$ we have
	\begin{align*}
		\beta_2(\|x\|)\geq -\bar V_-^2(x) &
		\geq \gamma(\|x\|),
	\end{align*}
	with $\beta_2\in\mathcal{K}$, where the upper bound is a direct consequence of Proposition~\ref{prop:bound} applied to $-\bar V_-^2$ on $\mathcal X_-$. 
	Moreover, by Lemma~\ref{lem:Vminus_bellman} we have
	\begin{align*}
		- \bar V_-^2(x_-) - (-\bar V_-^2(x))  = -L_2(x_-,u_-(x)) \leq 0,
	\end{align*}
	where $x_-:=x_-(x)$ is the optimal predecessor state, so that $f(x_-,u_-(x))=x$ 
	and feedback law $u_-$ is such that $(x_-,u_-(x))\in\mathcal{H}$.
	Consequently, $u_-(x)$ yields a stable closed-loop system backward in time, though this does not yet prove asymptotic stability. However, Constraint~\eqref{eq:ocp_bwd_tc} enforces the missing attractivity condition. Consequently, $u_-(x)$ yields an asymptotically stable closed-loop system.
	
	If the closed-loop system is additionally $\mathcal{KL}$-stable, then by~\cite[Theorem~6.2]{Gruene2014a} there exists a Lyapunov function $\mathcal{V}_-$ satisfying
	\begin{align*}
		\alpha_4(\|x\|) \geq \mathcal{V}_-(x) &\geq \alpha_5(\|x\|), \\
		\mathcal{V}_-(x_-) - \mathcal{V}_-(x) &\leq -\alpha_6(\|x\|),
	\end{align*}
	with $\alpha_i\in\mathcal{K}$, $i=4,5,6.$
		
	Let us rotate $\ell(x,u)$ using $\lambda_4(x) := \lambda_2(x) - \mathcal{V}_-(x)$ to get
	\begin{align*}
		L_4(x,u) &= L_2(x,u) - \mathcal{V}_-(x) + \mathcal{V}_-(f(x,u)).
	\end{align*}
	While in general $-\mathcal{V}_-(x) + \mathcal{V}_-(f(x,u))\ngeq0$ 
	such that we cannot prove strict dissipativity, we can prove that dissipativity holds along optimal trajectories of~\eqref{eq:ocp_bwd}, as
	\begin{align*}
		L_4(x_-,u_-(x)) &\geq L_2(x_-,u_-(x)) + \alpha_6(\|x\|) \geq \alpha_6(\|x\|).
	\end{align*}
	Consequently, because $\lambda_4(x)$ is continuous at $x=0$, we have that $-\bar V_-^4$ is a Lyapunov function because
	\begin{align*}
		\beta_4(\|x\|) \geq -\bar V_-^4(x) &\geq \gamma(\|x\|), \\
		-\bar V_-^4(x_-) + \bar V_-^4(x) &= -L_4(x_-,u_-(x)) \leq -\alpha_6(\|x\|),
	\end{align*}
	such that $u_-(x)$ is asymptotically stabilizing backward in time. Note that the upper bound is a direct consequence of Proposition~\ref{prop:bound} applied to $-\bar V_-^4$ on $\mathcal X_-$. 
	$\hfill\qed$
\end{pf}

Under a closed-graph condition on the optimal successor relation, the stability conclusion can be strengthened to $\mathcal{KL}$-stability.
\begin{Lemma}
	\label{lem:KL_stability}
	Suppose that the assumptions of Theorem~\ref{thm:2str_diss_implies_str_diss} hold. Assume that the optimal successor relation induced by $u_+$ has a closed graph. Then the system in closed loop with $u_+$ is $\mathcal{KL}$-stable on $\mathcal X_+$.
\end{Lemma}
\begin{pf}
	Theorem~\ref{thm:2str_diss_implies_str_diss} gives Lyapunov stability and attractivity of the origin. 
	By Assumption~\ref{ass:regularity}~\textit{(ii)}, for any $\xi>0$ all optimal trajectories starting in $\mathcal X_+\cap\mathcal B_\xi$ remain in the common compact set $\mathcal X_+^\xi$.
	
	Suppose that attractivity is not uniform on $\mathcal X_+\cap\mathcal B_\xi$.
	Then there exist $\varepsilon>0$, initial states $x_0^j\in\mathcal X_+\cap\mathcal B_\xi$, and integers $k_j\to\infty$ such that the corresponding optimal trajectories satisfy $\|x_{k_j}^j\|\geq\varepsilon$. 
	Compactness and a diagonal extraction yield a subsequence such that, for every fixed $k\geq0$, $x_k^j\to\bar x_k$.
	The closed-graph property implies that $\{\bar x_k\}_{k\geq0}$ is a closed-loop trajectory in $\mathcal X_+$.
	By uniqueness of the optimal state trajectory, it is the optimal trajectory starting from $\bar x_0$, and therefore $\bar x_k\to0$.
	
	By Lyapunov stability 
	every optimal trajectory starting with $\|x_0\|<\delta(\varepsilon)$ satisfies $\|x_k\|<\varepsilon$ for all $k\geq0$.
	Choose $K$ such that $\|\bar x_K\|<\delta(\varepsilon)/2$.
	For all sufficiently large $j$, $\|x_K^j-\bar x_K\|<\delta(\varepsilon)/2$, such that $\|x_k^j\|<\varepsilon$ for all $k\geq K$, which contradicts $\|x_{k_j}^j\|\geq\varepsilon$.
	Thus attractivity is uniform on every bounded subset of $\mathcal X_+$.
	
	The common compact bounds supplied by Assumption~\ref{ass:regularity}~\textit{(ii)}, together with Lyapunov stability, give a function $\alpha\in\mathcal K_\infty$ such that
	\begin{align*}
		\|x_k^+\| \leq\alpha(\|x_0^+\|), && \forall\,k\geq0.
	\end{align*}
	Together with uniform attractivity, this implies $\mathcal{KL}$-stability by \cite[Proposition~2.2]{Kellett2005}.
	$\hfill\qed$
\end{pf}

\begin{Remark}
	\label{rem:discuss_main_thm}
	With Theorem~\ref{thm:2str_diss_implies_str_diss} it was not possible to prove that two-storage strict dissipativity implies strict dissipativity except for optimal control inputs, though we will prove the converse, i.e., strict dissipativity implies two-storage strict dissipativity, in Theorem~\ref{thm:strict_diss_implies_2_storage_strict_diss}. 
	Furthermore, under the additional assumption that $V_\oplus$ be bounded from above on compact subsets of $\mathcal{X}_h$, if two-storage strict dissipativity holds, one can also select $\lambda_1(x)=-V_\ominus(x)$, $\lambda_2(x)=-V_\oplus(x)$, which entails that $\bar V_\oplus^1(x)=-\bar V_\ominus^2(x)$, such that the same Lyapunov function proves stability (though not asymptotic stability) forward in time for $u_+$ and backward in time for $u_-$. The difficulty stems from the fact that in general one cannot exclude $(u_+(x),0)\in \xi_\ominus(f(x,u_+(x)))$. If that were not the case~\eqref{eq:lyap_decrease0} would hold with a strict inequality and $\bar V_+^1$ would be a Lyapunov function.
	We suspect that, with a better-designed rotation, e.g., $\lambda_3(x) = \lambda_1(x) + \phi(\mathcal{V}_+(x))$, $\phi\in\mathcal{K}$, one might actually be able to prove the equivalence of two-storage strict dissipativity and strict dissipativity under the standing assumptions. That is the case in the linear-quadratic setting, as proven in~\cite{Zanon2025a}. 
	However, the investigation of this equivalence in the general nonlinear case is left for future research.
\end{Remark}
As we will prove next, if $u_+(x)$ is stabilizing forward in time, the corresponding inverse dynamics are unstable backward in time, provided they are well defined; and vice versa for $u_-(x)$.
\begin{Lemma}
	Let $f_-,f_+:\mathcal X\to\mathcal X$ satisfy $f_-(0)=f_+(0)=0$. 
	Suppose that there exists $x\neq0$ such that $ f_+^k(x)\to0 $ and $ f_-^k(f_+^k(x))=x$, $\forall \, k\in\mathbb N$. 
	Then the origin is not Lyapunov stable for the system $ y_{k+1}=f_-(y_k)$.
\end{Lemma}
\begin{pf}
	Set $\varepsilon:=\|x\|/2$. Given any $\delta>0$, choose $N$ such that $ \|f_+^N(x)\|<\delta$. 
	With $y_0=f_+^N(x)$, one has $ f_-^N(y_0)=f_-^N(f_+^N(x))=x, $ and hence $ \|f_-^N(y_0)\|=\|x\|>\varepsilon$. 
	Thus the origin is not Lyapunov stable for $f_-$. $\hfill\qed$
\end{pf}
\begin{Remark}
	\label{rem:u_minus_antistable}
	The fact that $u_-(x)$ is asymptotically stabilizing backward in time is especially well-known for the linear-quadratic case and has important implications. 
	Indeed, it is related to the stabilizing solution of the anticausal Riccati equation, which, under some conditions, corresponds to the antistabilizing solution of the Riccati equation~\cite{Ionescu1996}. That case has been studied in the context of economic MPC in~\cite{Zanon2025,Zanon2025a}.
\end{Remark}

We turn now our attention to comparing two-storage strict dissipativity to standard strict dissipativity, and prove next that the first is implied by the second.
\begin{Theorem}
	\label{thm:strict_diss_implies_2_storage_strict_diss}
	Assume that there exists a storage function $\lambda(x)$ such that strict dissipativity holds for all bounded $x\in\mathcal{X}_h$. 
	Then two-storage strict dissipativity holds for all bounded $x\in\mathcal{X}_h$.
\end{Theorem}
\begin{pf}
	Strict dissipativity entails the existence of $\lambda_1(x)=\lambda(x)$ such that, for all bounded $x\in\mathcal{X}_h$ we have
	\begin{align*}
		L_1(x,u) = \ell(x,u) + \lambda(x) - \lambda(f(x,u)) \geq \rho(\|x\|),
	\end{align*}
	with $0\leq\rho(\|x\|)\in\mathcal{PD}$. Then, we select $\lambda_2(x)=\lambda_1(x)-\rho(\|x\|)$ to obtain
	\begin{align*}
		L_2(x,u) &= L_1(x,u) - \rho(\|x\|) + \rho(\|f(x,u)\|) \\
		&\geq L_1(x,u) - \rho(\|x\|) \geq0.
	\end{align*}
	$\hfill\qed$
\end{pf}

We prove next that if two-storage strict dissipativity holds on $\mathcal{X}_+$, then, under an additional boundedness condition on $V_\oplus$, it also holds on $\mathcal{X}_h$. Moreover, we will further prove that two-storage strict dissipativity is also necessary for asymptotic stability whenever $V_+$ is bounded on bounded subsets of $\mathcal{X}_+$. To that end, we introduce first the next lemma.

\begin{Lemma}
	\label{lem:equal_V_means_equal_u}
	Suppose that Assumptions~\ref{ass:regularity} and~\ref{ass:relaxation} hold.
	Assume that dissipativity holds for every bounded $x\in\mathcal{X}_h$, and 
	$\infty>V_\oplus(\tilde x)=V_\ominus(\tilde x)>-\infty$ for some $\tilde x$. 
	Then, there exists at least one trajectory $\tilde x_k$, $\tilde u_k$, $\tilde z_k$ such that $\tilde x_j=\tilde x$ for some $j$ and
	\begin{align*}
		&(\tilde u_k,\tilde z_k) \in \xi_\oplus(\tilde x_k)\cap\xi_\ominus(f(\tilde x_k,\tilde u_k)+\tilde z_k), \\
		&V_\oplus(\tilde x_k)=V_\ominus(\tilde x_k).
	\end{align*}
\end{Lemma}
\begin{pf}
	We observe that, by optimality of $V_\oplus$ and $V_\ominus$,
	\begin{align*}
		V_\ominus(f(\tilde x,\tilde u)+\tilde z) &\geq -\ell(\tilde x,\tilde u)-p(\tilde x,\tilde u,\tilde z) + V_\ominus(\tilde x) \\
		&= -\ell(\tilde x,\tilde u)-p(\tilde x,\tilde u,\tilde z) + V_\oplus(\tilde x)\\
		&=V_\oplus(f(\tilde x,\tilde u)+\tilde z),
	\end{align*}
	for all $(\tilde u,\tilde z) \in \xi_\oplus(\tilde x)$. 
	Additionally, since $f(\tilde x,\tilde u)+\tilde z\in\mathcal{X}_h$, dissipativity entails the reverse inequality, such that
	\begin{align*}
		V_\oplus(f(\tilde x,\tilde u)+\tilde z) = V_\ominus(f(\tilde x,\tilde u)+\tilde z),
	\end{align*}
	and, consequently, 
	\begin{align*}
		(\tilde u,\tilde z) \in \xi_\oplus(\tilde x) \quad &\Longrightarrow \quad (\tilde u,\tilde z) \in \xi_\ominus(f(\tilde x,\tilde u)+\tilde z). 
	\end{align*}
	
	For some $\hat  x\in\mathcal{X}_h$, let us select $\hat x_-\in\mathcal{X}_h$ as  the optimal predecessor state, such that $f(\hat  x_-,\hat  u_-)+\hat z_-=\hat  x$, with 
	$(\hat  u_-,\hat z_-) \in \xi_\ominus(\hat x)$. 
	Then,
	using optimality of $V_\oplus$ and $V_\ominus$, if $V_\ominus(\hat  x)=V_\oplus(\hat  x)$, we further have
	\begin{align*}
		\hspace{4em}&\hspace{-4em}V_\ominus(\hat  x_-) -\ell(\hat  x_-,\hat  u_-)-p(\hat  x_-,\hat  u_-,\hat  z_-) \\
		&=  V_\ominus(\hat  x) \\
		&= V_\oplus(\hat  x) \\
		&\geq -\ell(\hat  x_-,\hat u_-)-p(\hat  x_-,\hat  u_-,\hat  z_-) + V_\oplus(\hat  x_-).
	\end{align*}
	However, dissipativity entails that $V_\oplus(\hat  x_-)\geq V_\ominus(\hat  x_-)$, such that $V_\ominus(\hat x_-) =  V_\oplus(\hat x_-)$, and $(\hat u_-,\hat z_-)\in \xi_\oplus(\hat x_-)$.
	
	By propagating inductively this reasoning forward and backward in time, we obtain the claim, where the forward recursion satisfies $\tilde x_{j+1} \in f(\tilde x_j,u_\oplus(\tilde x_j))+z_\oplus(\tilde x_j)$. 
	$\hfill\qed$
\end{pf}

\begin{Lemma}
	\label{lem:2storage_strict_diss_Xh}
	Suppose that Assumptions~\ref{ass:regularity} and~\ref{ass:relaxation} hold. 
	Assume that two-storage strict dissipativity holds for all bounded $x\in\mathcal{X}_+$, and that $V_\oplus$ is bounded from above on compact sets. 
	Then, two-storage strict dissipativity holds for all bounded $x\in\mathcal{X}_h$.
\end{Lemma}
\begin{pf}
	The assumptions entail that dissipativity holds for all bounded $x\in\mathcal{X}_h$ and, by Lemma~\ref{lem:boundedness_implies_diss}, $\lambda_1(x)=-V_\ominus(x)$ satisfies dissipativity and $V_\oplus(x)-V_\ominus(x)\geq0$ for all bounded $x\in\mathcal{X}_h$. Moreover, using the arguments of Lemma~\ref{lem:boundedness_implies_diss}, boundedness from above yields that $\lambda_2(x)=-V_\oplus(x)$ also satisfies dissipativity.
	
	We exploit next Assumption~\ref{ass:regularity}, which states that $B_\delta\cap\mathcal X_h\subseteq\mathcal X_+$. 
	Suppose by contradiction that there exists $\tilde x\neq0$ such that $V_\oplus(\tilde x)=V_\ominus(\tilde x)$. 
	By Lemma~\ref{lem:equal_V_means_equal_u}, there exists an optimal trajectory $\{\tilde x_k,\tilde u_k,\tilde z_k\}_{k\in\mathbb Z}$, with $\tilde x_0=\tilde x$, such that $V_\oplus(\tilde x_k) = V_\ominus(\tilde x_k)$, $\forall \, k\in\mathbb Z$.
	Since $\tilde x_k\to0$ as $k\to-\infty$, there exists $\bar k<0$ such that
	\begin{align*}
		\tilde x_k\in B_\delta\cap\mathcal X_h\subseteq\mathcal X_+,
		&& \forall\, k\leq\bar k.
	\end{align*}
	For any $k\leq\bar k$, Assumption~\ref{ass:relaxation} yields $V_\oplus(\tilde x_k)=V_+(\tilde x_k)$. Consequently, the unique unrelaxed optimal state trajectory starting from $\tilde x_k$ is also the unique optimal state trajectory of the relaxed problem.
	
	Lyapunov stability of the origin under feedback law $u_+$ entails that, $\forall \ \epsilon>0, \ \exists \ \delta(\epsilon)$ such that $\|x_{\hat k}\| < \delta(\epsilon) \implies \|x_{k}\| < \epsilon$ for all $k\geq \hat k$. We can select $\epsilon>0$ small enough such that $\epsilon<\|\tilde x\|$. 
	Since $\tilde x_k\to0$ as $k\to-\infty$, we can select $\hat k\leq\bar k$ sufficiently negative such that $\|\tilde x_{\hat k}\|<\delta(\epsilon)$. The forward continuation from $\tilde x_{\hat k}$ is the optimal closed-loop trajectory, i.e., the one obtained using feedback law $u_+$. By construction, we have $\tilde x_0=\tilde x$, and Lyapunov stability implies $\|\tilde x\|<\epsilon$, contradicting $\|\tilde x\|>\epsilon$.

	Consequently, $V_\oplus(x)>V_\ominus(x)$, $\forall \, x\in\mathcal X_h\setminus\{0\}$. 
	Since $V_\oplus-V_\ominus$ is lower semicontinuous and $\mathcal X_h$ is closed, there exists $\gamma\in\mathcal{PD}$ such that
	\begin{align*}
		V_\oplus(x)-V_\ominus(x)\geq\gamma(\|x\|), && \forall x\in\mathcal X_h.
	\end{align*}
	Therefore, $\lambda_1=-V_\ominus$ and $\lambda_2=-V_\oplus$ satisfy two-storage strict dissipativity on $\mathcal X_h$.
	$\hfill\qed$
\end{pf}

\begin{Remark}
	The additional boundedness requirement on $V_\oplus$ is automatically satisfied in case the constraint set $\mathcal{H}$ is compact. Otherwise, it is partly due to the specific relaxation adopted in~\eqref{eq:ocp_fwd_relaxed}. Indeed, this does not allow us to symmetrize the proof in Lemma~\ref{lem:boundedness_implies_diss}, as one cannot guarantee the existence of a uniformly bounded feasible input selection over bounded subsets of $\mathcal{X}_h$. If such a selection were available, boundedness of $f$ and $\ell$ on bounded sets would yield the upper bound $V_\oplus(x) \leq \ell(x,u)+p(x,u,-f(x,u))$. 
	
	To circumvent that issue, one could further relax the initial condition, e.g., $x_0=\hat x+z_{-1}$, and penalize $z_{-1}$ with some penalty $p_0(\hat x,z_{-1})$ having properties similar to those of $p$. This would provide a direct mechanism to bound $V_\oplus$ on bounded subsets of $\mathcal X_h$ by selecting $z_{-1}=-\hat x$, at the price of strengthening the exactness requirements for the relaxed problem. Since such an extension is not needed for the developments of this paper, we do not investigate it further.
	
	Alternatively, in Assumption~\ref{ass:regularity} one could require that, for every $\xi>0$, there exists a compact set $\mathcal U_+^\xi\subset\mathbb R^{n_u}$ such that, for every $x\in\mathcal X_h\cap\mathcal B_\xi$, there exists $u\in\mathcal U_+^\xi$ satisfying $(x,u)\in\mathcal H$, rather than imposing this requirement only for every $x\in\mathcal X_+\cap\mathcal B_\xi$. Then, the arguments used in the proof of Theorem~\ref{thm:2str_diss_implies_str_diss} would establish boundedness of $V_\oplus$ on bounded subsets of $\mathcal X_h$.
\end{Remark}
\begin{Remark}
	The upper boundedness assumption on $V_\oplus$ in Lemma~\ref{lem:2storage_strict_diss_Xh} can be omitted by choosing, for any constant $c>0$, $\lambda_2(x)=\max\{-V_\oplus(x),-V_\ominus(x)-c\}$.
	The maximum preserves the dissipation inequality, while $\lambda_1(x)-c\leq\lambda_2(x)\leq\lambda_1(x)$ ensures boundedness on bounded subsets, and $\lambda_2(x)=-V_\oplus(x)$ near the origin ensures continuity and normalization there. The unchanged value-function gap argument yields $ \lambda_1(x)-\lambda_2(x) =\min\{V_\oplus(x)-V_\ominus(x),c\} \geq\min\{\gamma(\|x\|),c\}, $ whose right-hand side is positive definite.
\end{Remark}

\begin{Theorem}
	\label{thm:necessity}
	Suppose that Assumptions~\ref{ass:regularity} and~\ref{ass:relaxation} hold.
	Assume that $u_+(x)$ is asymptotically stabilizing with bounded cost for all bounded $x\in\mathcal{X}_+$. Then, two-storage strict dissipativity must hold on bounded subsets of $\mathcal{X}_+$. 
\end{Theorem}
\begin{pf}
	Because by assumption $V_+(x)>-\infty$ for all bounded $x\in\mathcal{X}_+$, by Lemma~\ref{lem:boundedness_implies_diss} there exists a storage function $\lambda(x)$ which yields dissipativity and, consequently, 
	$V_\ominus(x)<\infty$ for all bounded $x\in\mathcal{X}_h$. 
	
Select $\lambda_1(x):=-V_\ominus(x)$, such that $\bar V_+^1(x)=V_+(x)-V_\ominus(x)$. Then, select $\lambda_2(x)=\lambda_1(x)-\lambda_{1,2}(x)$, with $\lambda_{1,2}(x)=\bar V_+^1(x)$, i.e., $\lambda_2(x)=-V_+(x)$. By Lemma~\ref{lem:boundedness_implies_diss}, $\lambda_1$ is bounded on bounded subsets of $\mathcal X_+$, while boundedness of $\lambda_2$ follows from the bounded-cost assumption. Moreover, both $\lambda_1(x)$ and $\lambda_2(x)$ satisfy dissipativity on $\mathcal X_+$ as an immediate consequence of their definition. We are therefore left with proving that $\lambda_1(x)\geq\lambda_2(x)+\gamma(\|x\|)$, with $\gamma\in\mathcal{PD}$, which can be done following the same arguments as in Lemma~\ref{lem:2storage_strict_diss_Xh}. Hence, two-storage strict dissipativity holds on $\mathcal X_+$.
$\hfill\qed$
\end{pf}

\begin{Remark}
	Boundedness of the optimal cost is necessary in general, as proven by two counterexamples in~\cite{Muller2013a}. Under Assumptions~\ref{ass:regularity} and~\ref{ass:relaxation}, however, $V_+$ is already bounded on bounded subsets of $\mathcal X_+$. Nevertheless, for the sake of clarity, we write this assumption explicitly in Theorem 3.15.
\end{Remark}

\begin{Remark}
	These results establish that two-storage strict dissipativity is not a stronger assumption than strict dissipativity, it is implied by the latter, and, under a bounded-cost condition, it is not only sufficient but also necessary for asymptotic stability. Whether the two assumptions are actually equivalent under the standing assumptions is an open question. However, checking two-storage strict dissipativity might be easier than checking strict dissipativity. More details about the linear-quadratic setting, in which the two are actually equivalent, are discussed in~\cite{Zanon2025a}. 
\end{Remark}

\begin{Remark}
	The result of Theorem~\ref{thm:2str_diss_implies_str_diss} also connects in a rather direct way to whether the system is optimally operated at steady state according to Definition~\ref{def:optimal_steady_state} (see also~\cite{Angeli2012a,Mueller2015} for an alternative definition), as it states that the total cost of a round trip that visits a nonzero state must be positive. This fact becomes even more evident in the proof of Theorem~\ref{thm:necessity}. In order to further clarify this aspect, suppose that there exists a bounded periodic trajectory $\hat x_k$ with negative cost, then we immediately know that there exists at least one $\tilde x$ (e.g., $\tilde x=\hat x_k$ for some $k$) for which $V_\oplus(\tilde x) - V_\ominus(\tilde x) = -\infty$, as one can repeat the loop arbitrarily many times before converging to the optimal steady state, such that two-storage strict dissipativity, and actually even dissipativity cannot hold. In the light of Theorem~\ref{thm:strict_diss_implies_2_storage_strict_diss} we immediately have that also strict dissipativity cannot hold. 
	Another trivial situation in which we immediately obtain the obvious conclusion one would expect is the one in which $\ell(x,u)=0$ and $(\mathcal{X}_+ \cap \mathcal{X}_-) \setminus \{0\}\neq \emptyset$, as in this case $V_+(x)=V_-(x)=0$ for all $x\in \mathcal{X}_+\cap\mathcal{X}_-$.
\end{Remark}

\section{Finite-Horizon OCP with Terminal Cost}
\label{sec:finite_horizon}

In this section, we consider the finite-horizon case, where one solves in a receding horizon fashion Problem~\eqref{eq:mpc}, which we recall here for the reader's convenience:
\begin{subequations}
	\begin{align*}
		V_N(\hat x) = \min_{x,u} \ & \sum_{k=0}^{N-1} \ell(x_k , u_k) + V_\mathrm{f}(x_N) \\
		\mathrm{s.t.} \ & x_0 = \hat x, \\
		&x_{k+1} = f(x_k,u_k), && k\in\mathbb{I}_0^{N-1},  \\
		&h(x_k,u_k) \leq 0, && k\in\mathbb{I}_0^{N-1},  \\
		&h_\mathrm{f}(x_N) \leq 0,
	\end{align*}
\end{subequations}
with optimal feedback law $u_N^\star(\hat x)=u_0^\star,$ where we denote the optimal solution as $x^N=(x_0^\star, \ldots,x_N^\star)$, $u^N=(u_0^\star, \ldots,u_{N-1}^\star)$. In case Problem~\eqref{eq:mpc} is infeasible we define $V_N(x)=\infty$, and we denote the domain of~\eqref{eq:mpc} as $\mathcal{X}_N:= \{ \, x \,|\, V_N(x) < \infty \,\}$ and the terminal constraint set as $\mathcal{X}_\mathrm{f} := \{ \, x \,|\, h_\mathrm{f}(x) \leq 0 \,\}$.

We discuss next how a suitable choice of terminal cost $V_\mathrm{f}$ can ensure asymptotic stability with a finite horizon. We first prove that two-storage strict dissipativity is sufficient to prove that the terminal cost and constraint set proposed in~\cite{Amrit2011a} do yield asymptotic stability. Then, we discuss how alternative choices of terminal cost also yield asymptotic stability, provided that the prediction horizon is long enough.

In order to prove our results, we need to introduce the following simplifying assumption. 
While it might be possible to relax this assumption, in order to avoid further technicalities we prefer to keep it.
\begin{Assumption}
	\label{ass:Bellman_regularity}
	Functions $V_\oplus$, $V_\ominus$ are continuous, all solutions $V_\dagger$ of the Bellman equation~\eqref{eq:bellman} considered below are continuous, their domain $\mathcal{\bar X}_\dagger := \{ \, x \,|\, |V_\dagger(x)|<\infty \,\}$ is closed and the corresponding optimal feedback $u_\dagger(x)$ is such that $f(x,\bar u)$ with $\bar u\in u_\dagger(x)$ has a closed graph. Finally, Assumption~\ref{ass:relaxation} holds also with respect to $V_\dagger$, i.e., $V_\circleddagger(x)=V_\dagger(x)$ for all $x\in\mathcal{\bar X}_\dagger$.
\end{Assumption}

\begin{Theorem}
	\label{thm:stability_amrit}
	Suppose that Assumptions~\ref{ass:regularity},~\ref{ass:relaxation}, and~\ref{ass:Bellman_regularity} hold. Assume that for all $N\geq1$ the value functions $V_N$ are continuous on $\mathcal{X}_N$ and the corresponding optimal feedback $u_N^\star(x)$ is such that $f(x,\bar u)$ with $\bar u\in u_N^\star(x)$ has a closed graph. 
	Assume that  two-storage strict dissipativity holds and there exists a compact terminal region $\mathcal{X}_\mathrm{f}$ containing the origin in its interior; the terminal cost satisfies $V_\mathrm{f}(0)=0$, and $V_\mathrm{f}$ is continuous in $\mathcal{X}_\mathrm{f}$; and there exists a terminal control law $\kappa_\mathrm{f}$ such  that $x\mapsto f(x,\kappa_\mathrm{f}(x))$ has a closed graph and for all $x\in\mathcal{X}_\mathrm{f}$ we have
	\begin{align}
		\label{eq:terminal_cost_amrit}
		&V_\mathrm{f}(f(x,\kappa_\mathrm{f}(x))) - V_\mathrm{f}(x) \leq -\ell(x,\kappa_\mathrm{f}(x)), \\
		&h(x,\kappa_\mathrm{f}(x))\leq0, \\
		&f(x,\kappa_\mathrm{f}(x)) \in \mathcal{X}_\mathrm{f}.
	\end{align}
	Assume also that, for every $N\geq1$ and every compact set $\mathcal{\bar X}\subseteq\mathcal X_N$, there exists a compact set $\mathcal C_N^{\mathcal{\bar X}}\subseteq\mathcal X_h$ such that every closed-loop trajectory generated by $u_N^\star$ with $x_0\in\mathcal{\bar X}$ satisfies $x_k\in\mathcal C_N^{\mathcal{\bar X}}$, for all $k\geq0.$

	Then, the economic MPC OCP~\eqref{eq:mpc} yields a feedback control law $u_N^\star(x)$ which is asymptotically stabilizing for all $x\in\mathcal{X}_N$ for all $N$.
\end{Theorem}
\begin{pf}
	Because $V_\oplus$ is continuous by Assumption~\ref{ass:Bellman_regularity}, it is bounded on compact subsets of $\mathcal X_h$. Hence, Lemma~\ref{lem:2storage_strict_diss_Xh} applies, and two-storage strict dissipativity holds with $\lambda_1=-V_\ominus$ and $\lambda_2=-V_\oplus$. 
	
	We observe that, by definition, set $\mathcal{X}_\mathrm{f}$ is positive invariant under feedback control law $\kappa_\mathrm{f}$. 
	By rearranging the terms in~\eqref{eq:terminal_cost_amrit} we have
	\begin{align*}
		V_\mathrm{f}(x) \geq \ell(x,\kappa_\mathrm{f}(x))+V_\mathrm{f}(f(x,\kappa_\mathrm{f}(x))).
	\end{align*}
	We prove next that $V_\mathrm{f}(x) \geq V_+(x)$. To that end, we first prove that $\lim_{k\to\infty}  V_\mathrm{f}( x_k)= 0$.
	Consider any closed-loop trajectory $x_{k+1}=f(x_k,\kappa_\mathrm{f}(x_k))$. 
	Define the limit set
	\begin{align*} 
		\Omega(x_0) := \left\{ 	\bar x\in\mathbb R^{n_x} 	\middle| \,
		\exists \, \{k_j\}_{j\in\mathbb N}, \, k_j\to\infty, \, 	x_{k_j}\to \bar x 	
		\right\}.
	\end{align*}
	Since compact set $\mathcal{X}_\mathrm{f}$ is positive invariant under feedback control law $\kappa_\mathrm{f}$, set $\Omega(x_0)$ is nonempty and compact.
	
	Using $\lambda_1(x)=-V_\ominus(x)$ and $\lambda_2(x)=-V_\oplus(x)$ we have
	\begin{align*}
		\bar V_\mathrm{f}^1(x)&=V_\mathrm{f}(x)-V_\ominus(x), &\bar V_\mathrm{f}^2(x)&=V_\mathrm{f}(x)-V_\oplus(x),
	\end{align*}
	and, for $i=1,2$,
	\begin{align*}
		\bar V_\mathrm{f}^i(f(x,\kappa_\mathrm{f}(x))) -\bar V_\mathrm{f}^i(x) &\leq-L_i(x,\kappa_\mathrm{f}(x))\leq0.
	\end{align*}
	Since $\mathcal X_\mathrm{f}$ is compact and $\bar V_\mathrm{f}^i$ is continuous, both functions are bounded from below along the closed loop. Consequently, $\lim_{k\to\infty}\bar V_\mathrm{f}^i(x_k) =c_i,$ $i=1,2$.
	By continuity, $\bar V_\mathrm{f}^i(\bar x)=c_i$ for every $\bar x\in\Omega(x_0)$.
	
	Closedness of the closed-loop successor graph implies that, for every $\bar x_0\in\Omega(x_0)$, there exists a forward trajectory $\{\bar x_k\}_{k\geq0}\subseteq\Omega(x_0)$ satisfying $\bar x_{k+1}=f(\bar x_k,\kappa_\mathrm{f}(\bar x_k))$. 
	Since $\bar V_\mathrm{f}^2$ is constant along this trajectory, we have $L_2(\bar x_k,\kappa_\mathrm{f}(\bar x_k)) =0,$ for all $k\geq0$, such that, by uniqueness of the optimal state trajectory, $\{\bar x_k\}_{k\geq0}$ is optimal for $V_\oplus(\bar x_0)$. 
	Consequently, $\bar x_k\to0$, and closedness of $\Omega(x_0)$ gives $0\in\Omega(x_0)$.
	
	Since $\bar V_\mathrm{f}^i(0)=0$, we have $c_i=0$, and, for all $\bar x\in\Omega(x_0)$,
	\begin{align*}
		V_\oplus(\bar x)-V_\ominus(\bar x) &=\bar V_\mathrm{f}^1(\bar x) -\bar V_\mathrm{f}^2(\bar x)=0.
	\end{align*}
	By two-storage strict dissipativity, $\bar x=0$, such that $\Omega(x_0)=\{0\}$ and $x_k\to0$.
	
	Leveraging this result, the inequality $V_\mathrm{f}(x) \geq V_+(x)$ can be obtained using the arguments used in~\cite[Lemma~11]{Amrit2011a}, when rotating the cost using $\lambda(x)=-V_+(x)$.
	
	Let us rotate the cost using $\lambda_1(x)=-V_\ominus(x)$, to obtain rotated stage cost $L_1$ and rotated terminal cost $\bar V_\mathrm{f}^1(x)=V_\mathrm{f}(x)+\lambda_1(x)$. 
	By using Proposition~\ref{prop:bound}, $V_\mathrm{f}(x) \geq V_+(x)$, $V_+(x)-V_\ominus(x)\geq \gamma(\|x\|)$, and the rotation of the terminal cost combined with~\eqref{eq:terminal_cost_amrit}, see also~\cite[Lemma~9]{Amrit2011a}, we obtain $\forall \, x\in\mathcal{X}_\mathrm{f}$
	\begin{align*}
		\beta_\mathrm{f}(\|x\|)\geq \bar V_\mathrm{f}^1(x) &\geq \gamma(\|x\|), \\
		\bar V_\mathrm{f}^1(f(x,\kappa_\mathrm{f}(x))) - \bar V_\mathrm{f}^1(x) &\leq -L(x,\kappa_\mathrm{f}(x))\leq0,
	\end{align*}
	such that $\bar V_\mathrm{f}^1$ is a Lyapunov function and $\kappa_\mathrm{f}$ is a stabilizing control law, though we have not yet proven that it is asymptotically stabilizing. 
	As we proved earlier, $\Omega(x_0)=\{0\}$, which establishes attractivity, i.e.,  $\lim_{k\to\infty}x_k=0$. 
	Hence, $\kappa_\mathrm{f}$ asymptotically stabilizes the origin.
	
	Consider next $N\geq1$. Appending to the optimal prediction the control given by $\kappa_\mathrm{f}$ and the corresponding successor state proves positive invariance of $\mathcal X_N$ and gives
	\begin{align}
		\label{eq:V_N_nonincrease}
		\bar V_N^i(f(x,u_N^\star(x))) -\bar V_N^i(x) &\leq-L_i(x,u_N^\star(x))\leq0,
	\end{align}
	for $i=1,2$, where $\bar V_N^1(x):=V_N(x)-V_\ominus(x)$ and $\bar V_N^2(x):=V_N(x)-V_\oplus(x)$.
		
	Since the terminal controller asymptotically stabilizes the origin and $V_\mathrm{f}\geq V_+$ on $\mathcal X_\mathrm{f}$, appending its trajectory to a feasible prediction gives $\mathcal X_N\subseteq\mathcal X_+$ and $V_N\geq V_+$ on $\mathcal X_N$. Consequently,
	\begin{align*}
		\bar V_N^1(x) &\geq V_+(x)-V_\ominus(x) \geq\gamma(\|x\|).
	\end{align*}
	Moreover, $\bar V_N^1$ is continuous and $\bar V_N^1(0)=0$.
	
	The trajectory compactness hypothesis using $\mathcal{\bar X}=\mathcal X_\mathrm{f}$ provides a common compact set containing all closed-loop trajectories starting in $\mathcal X_\mathrm{f}$. On this compact set, the lower bound $\gamma(\|x\|)$ is uniformly positive away from the origin. Continuity of $\bar V_N^1$ at the origin and~\eqref{eq:V_N_nonincrease} therefore imply Lyapunov stability.

	For an arbitrary $x_0\in\mathcal X_N$, consider the MPC closed-loop trajectory $\{x_k\}_{k\geq0}$ generated by $u_N^\star$. The trajectory compactness hypothesis makes
	the limit set $\Omega(x_0)$ nonempty and compact. The closed-graph property implies that $\Omega(x_0)\subseteq\mathcal X_N$ and that every point of $\Omega(x_0)$ lies on a closed-loop trajectory $\{\bar x_k\}_{k\geq0}\subseteq\Omega(x_0)$. Both rotated value functions are continuous on $\mathcal X_N$ and nonincreasing along the closed loop. Moreover, $\bar V_N^1\geq\gamma(\|x\|)$ and $\bar V_N^2=V_N-V_+\geq0$ on $\mathcal X_N$. Consequently, both functions are constant on $\Omega(x_0)$.
	
	Because $\bar V_N^2$ is constant on $\{\bar x_k\}_{k\geq0}$, the decrease inequality gives
	\begin{align*}
		L_2(\bar x_k,\bar u_k)&=0, &\bar u_k&\in u_N^\star(\bar x_k), &&\forall\,k\geq0,
	\end{align*}
	where $\bar x_{k+1}=f(\bar x_k,\bar u_k)$. 
	As in the preceding forward-optimality argument, uniqueness of the optimal state trajectory implies $\bar x_k\to0$ as $k\to\infty$, and hence $0\in\Omega(x_0)$. Since $\bar V_N^1$ is constant on $\Omega(x_0)$ and continuous at the origin, its value there is $\bar V_N^1(0)=0$. 
	The positive-definite lower bound then yields $\Omega(x_0)=\{0\}$ and $x_k\to0$. 
	Consequently, $u_N^\star$ asymptotically stabilizes the origin on $\mathcal X_N$ for every $N\geq1$.
	$\hfill\qed$
\end{pf}

The trajectory compactness condition in Theorem~\ref{thm:stability_amrit} holds immediately if $\mathcal X_h$ is compact, or, alternatively, if $V_+-V_\ominus$ has bounded sublevel sets as, for every compact $\mathcal{\bar X}\subseteq\mathcal X_N$, continuity and the decrease inequality give
	\begin{align*}
		V_+(x_k)-V_\ominus(x_k) \leq\bar V_N^1(x_k) \leq\max_{x\in\mathcal{\bar X}}\bar V_N^1(x) <\infty.
\end{align*}

While we have proven that two-storage strict dissipativity is sufficient for asymptotic stability when relying on standard terminal cost and constraints, we prove next that asymptotic stability can be obtained also using an alternative choice of terminal cost for long enough prediction horizons even in the absence of terminal constraints. 
To that end, we first prove a technical result about the solutions $V_\dagger$ of the Bellman Equation~\eqref{eq:bellman}.

In particular, we restrict our attention to solutions $V_\dagger$ of the Bellman equation~\eqref{eq:bellman} which also solve, for every $N$, including the limit as $N\to\infty$, the following OCP:
\begin{subequations}
	\begin{align}
		\bar V_\dagger(\hat x) = \inf_{x,u} \ & \sum_{k=0}^{N-1} L(x_k , u_k) + \bar V_\dagger(x_N) \hspace{-4em} \\
		\mathrm{s.t.} \ & x_0 = \hat x, \\
		&x_{k+1} = f(x_k,u_k), && k\in\mathbb{I}_0^{N-1},  \\
		&h(x_k,u_k) \leq 0, && k\in\mathbb{I}_0^{N-1}.
	\end{align}
	\label{eq:ocp_dagger_fwd}
\end{subequations}
As for OCP~\eqref{eq:ocp_fwd}, we can define a relaxation of~\eqref{eq:ocp_dagger_fwd}, yielding the value function $V_\circleddagger$.

\begin{Lemma}
	\label{lem:V_dagger}
	Suppose that Assumptions~\ref{ass:regularity},~\ref{ass:relaxation}, and~\ref{ass:Bellman_regularity} hold. 
	Assume that for every $\xi>0$ there exists a compact set $\mathcal X^\xi_\dagger$ such that, for every $x^\dagger_0\in\bar{\mathcal X}_\dagger\cap\mathcal B_\xi$, there exists a Bellman-optimal trajectory $\{x^\dagger_k,u^\dagger_k\}_{k\geq0}$ associated with $V_\dagger$ satisfying $x^\dagger_k\in\mathcal X^\xi_\dagger$, $\forall\,k\geq0$. 
	Then, for any solution $V_\dagger$ of the Bellman equation~\eqref{eq:bellman} with $V_\dagger(0)=0$ if there exists a positive definite function $\gamma$ such that
	\begin{align*}
		V_\dagger(x)&\geq V_\ominus(x)+\gamma(\|x\|), && \forall\, x\in\mathcal{\bar X}_\dagger,
	\end{align*}
	then,
	\begin{align*}
		\bar{\mathcal X}_\dagger &\subseteq \mathcal X_+, &
		V_\dagger(x) &\geq V_+(x),  && \forall\, x\in\bar{\mathcal X}_\dagger.
	\end{align*}
\end{Lemma}
\begin{pf}
	Let us define 
	\begin{align*}
		L(x,u) 
		&=\ell(x,u)  - V_\ominus(x) + V_\ominus(f(x,u)) \geq 0. 
	\end{align*}
	We prove next that if for some $x$
	\begin{align}
		\label{eq:assumption_Vdagger=V-}
		\bar V_\dagger(x) := V_\dagger(x) - V_\ominus(x) \geq \gamma(\|x\|),
	\end{align}
	then $V_\dagger(x)\geq V_+(x)$. 
	
	We observe that 
	by Proposition~\ref{prop:bound} there exists a function $\beta_\dagger\in\mathcal{K}$ such that
	\begin{align}
		\beta_\dagger(\|x\|)\geq\bar V_\dagger(x) \geq \gamma(\|x\|), \qquad \forall \, x\in\mathcal{\bar X}_\dagger.
		\label{eq:V_dagger_bounds}
	\end{align}
	Moreover, for all $x\in\mathcal{\bar X}_\dagger$ we have
	\begin{align*}
		\bar V_\dagger(f(x,u_\dagger(x))) - \bar V_\dagger(x) = -L(x,u_\dagger(x)) \leq 0,
	\end{align*}
	such that $u_\dagger(x)$ yields a stable closed-loop system, though this does not yet prove asymptotic stability. 
	
	Let $\{x_k^\dagger\}_{k\geq0}$ be a Bellman-optimal trajectory associated with
	$V_\dagger$, i.e., $x_{k+1}^\dagger=f(x_k^\dagger,u^\dagger(x_k^\dagger))$, and $\bar V_\dagger(x_k^\dagger) = L(x_k^\dagger,u^\dagger(x_k^\dagger)) + \bar V_\dagger(x_{k+1}^\dagger)$,
	such that 
	for all $k\geq0$ we have $\bar V_\dagger(x_{k+1}^\dagger) \leq \bar V_\dagger(x_k^\dagger)$,
	and the scalar sequence
	$\{\bar V_\dagger(x_k^\dagger)\}_{k\geq0}$ is monotone and convergent, i.e., 
	there exists a constant $c\geq0$ such that
	\begin{align}
		\lim_{k\to\infty} \bar V_\dagger(x_k^\dagger) = c. 
		\label{eq:rotated_value_limit}
	\end{align}
	
	Since $\mathcal X_\dagger^\xi$ is compact,
	$\mathcal{\bar X} := \operatorname{cl} \left\{ x_k^\dagger\;\middle|\;k\geq0 \right\}$
	is compact. 
	Define the limit set
	\begin{align*}
		\Omega(x_0^\dagger)
		:=
		\left\{
		\bar x\in\mathcal{\bar X}
		\;\middle|\;
		\exists \, k_j\to\infty 
		\text{ such that } x_{k_j}^\dagger\to \bar x
		\right\},
	\end{align*}
	which is nonempty and, due to compactness of $\mathcal{\bar X}$, also compact.
	
	We prove next that $\Omega(x_0^\dagger)$ is weakly invariant. 
	By definition, there exists a sequence $k_j\to\infty$ such that $x_{k_j}^\dagger\to \bar x$.
	Since $\mathcal{\bar X}$ is compact, the sequence
	$\{x_{k_j+1}^\dagger\}$ possesses a convergent subsequence. 
	We write $x_{k_j+1}^\dagger\to \bar x^+$. Since $k_j+1\to\infty$, then $\bar x^+\in\Omega(x_0^\dagger)$. 
	Moreover, $x_{k_j+1}^\dagger \in  f(x_{k_j}^\dagger,u^\dagger(x_{k_j}^\dagger))$ 
	for every $j$. Closedness of the graph of
	$f(\cdot,u^\dagger(\cdot))$ therefore yields 
	\begin{align}
		\bar x^+\in f(\bar x,u^\dagger(\bar x)).
		\label{eq:weak_invariance_omega}
	\end{align}
	
	Continuity implies that $\bar V_\dagger$ is constant on
	the whole limit set $\Omega(x_0^\dagger)$. To see that, let $\bar x\in\Omega(x_0^\dagger)$ and choose
	$k_j\to\infty$ such that $x_{k_j}^\dagger\to \bar x$.
	Then, using \eqref{eq:rotated_value_limit}, $\bar V_\dagger(\bar x) = \lim_{j\to\infty} \bar V_\dagger(x_{k_j}^\dagger) = c$.
	Since $\bar x\in\Omega(x_0^\dagger)$ was arbitrary, 
	\begin{align}
		\bar V_\dagger(\bar x) = c \qquad \forall \,\bar x\in\Omega(x_0^\dagger).
		\label{eq:value_constant_entire_omega}
	\end{align}
	
	Fix now any $\bar x_0\in\Omega(x_0^\dagger)$. By
	\eqref{eq:weak_invariance_omega}, there exists
	$\bar x_1\in\Omega(x_0^\dagger)$ and a Bellman-optimal control $\bar u_0$ such that $\bar x_1 = f(\bar x_0,\bar u_0)$ and $\bar V_\dagger(\bar x_0) = L(\bar x_0,\bar u_0) + \bar V_\dagger(\bar x_1)$.
	Since both $\bar x_0$ and $\bar x_1$ belong to $\Omega(x_0^\dagger)$,
	\eqref{eq:value_constant_entire_omega} yields $L(\bar x_0,\bar u_0)=0$. Repeating the invariance argument produces a forward Bellman-optimal trajectory $\{\bar x_r\}_{r\geq0}$
	contained entirely in $\Omega(x_0^\dagger)$ and satisfying $L(\bar x_r,\bar u_r) = 0$, $\forall r\geq0$.
	
	To extend the limiting trajectory backward in time, let $\bar x_0\in\Omega(x_0^\dagger)$ and choose $k_j\to\infty$ such that $x_{k_j}^\dagger\to \bar x_0$. For any fixed $m\geq1$, we have $k_j\geq m$ for all sufficiently large $j$. Hence, by compactness and possibly extracting a subsequence, there exist $\bar x_{-1},\ldots,\bar x_{-m}\in\Omega(x_0^\dagger)$ such that $x_{k_j-r}^\dagger\to\bar x_{-r}$, $r\in\mathbb{I}_1^m$. 
	Closedness of the Bellman-optimal successor relation yields Bellman-optimal transitions $\bar x_{-r+1}=f(\bar x_{-r},\bar u_{-r})$, $r\in\mathbb{I}_1^m$. 
	Since $\bar V_\dagger$ is constant on $\Omega(x_0^\dagger)$, these transitions satisfy $L(\bar x_{-r},\bar u_{-r})=0$, which implies that $(\bar u_{-r},0)$ is a backward-optimal one-step transition for $V_\ominus(\bar x_{-r+1})$. Let an optimal trajectory of OCP~\eqref{eq:ocp_bwd_relaxed} ending at $\bar x_{-m}$ be given. Appending the transitions $\bar x_{-m}\to\cdots\to\bar x_0$ preserves optimality and therefore yields an optimal backward trajectory ending at $\bar x_0$. By uniqueness of the optimal state trajectory of OCP~\eqref{eq:ocp_bwd_relaxed}, the states $\bar x_{-m},\ldots,\bar x_{-1}$ must coincide with the last $m$ predecessor states of the unique backward-optimal trajectory ending at $\bar x_0$. Since $m$ is arbitrary, the complete negative-time trajectory contained in $\Omega(x_0^\dagger)$ coincides with that backward-optimal trajectory. Consequently, $\lim_{r\to-\infty}\bar x_r=0$.
	Since $\bar V_\dagger$ is continuous and
	$\bar V_\dagger(\bar x_r)=c$ for all $r\in\mathbb Z$, one obtains
	\begin{align*}
		c &= \lim_{r\to-\infty} \bar V_\dagger(\bar x_r) = \bar V_\dagger(0) = 0.
	\end{align*}
	Using the positive-definite lower bound in~\eqref{eq:V_dagger_bounds} we deduce $\bar x=0$ for all $\bar x\in\Omega(x_0^\dagger)$, such that $\Omega(x_0^\dagger) = \{0\}$.
	
	Since the original trajectory is bounded and its only limit point is the origin, the entire trajectory converges to the origin. Hence, for every $x\in\bar{\mathcal X}_\dagger$, the Bellman-optimal trajectory associated with $V_\dagger$ is feasible for~\eqref{eq:ocp_fwd}. 
	Consequently, $\bar{\mathcal X}_\dagger\subseteq\mathcal X_+$.
	
	Moreover, iterating the Bellman equality along this trajectory yields 
	$V_\dagger(x_0^\dagger) = \sum_{k=0}^{M-1}\ell(x_k^\dagger,u_k^\dagger) + V_\dagger(x_M^\dagger)$.
	Since $x_M^\dagger\to0$, continuity of $V_\dagger$ at the origin,
	$V_\dagger(0)=0$, and feasibility of this trajectory for~\eqref{eq:ocp_fwd}, imply $V_+(x)\leq V_\dagger(x)$. 
$\hfill\qed$
\end{pf}

Finally, in order to prove our result, we will rely on the following lemma, which should be well-known, but for which we could not find a proof. Since the proof is very simple, we provide it next.
\begin{Lemma}
	\label{lem:value_iteration_bound}
	Assume that $V_{N+1}(x) \geq V_N(x)$ for all $x$. Then, $V_{N+2}(x) \geq V_{N+1}(x)$ for all $x$.
\end{Lemma}
\begin{pf}
	We observe that, for all $x$
	\begin{align*}
		V_{N+2}(x) &= \ell(x,u_{N+2}^\star(x)) + V_{N+1}(f(x,u_{N+2}^\star(x))) \\
		&\geq \ell(x,u_{N+2}^\star(x)) + V_{N}(f(x,u_{N+2}^\star(x))) \\
		&\geq \ell(x,u_{N+1}^\star(x)) + V_{N}(f(x,u_{N+1}^\star(x))) \\
		&= V_{N+1}(x),
	\end{align*}
	where we used optimality of the one-step-ahead problem with terminal cost $V_N$ to obtain the second inequality, and the first inequality stems from the assumption $V_{N+1}(x)\geq V_N(x)$.  $\hfill\qed$
\end{pf}

Having established the necessary technical results above, we introduce next an important assumption on the terminal cost and two technical assumptions that guarantee sufficient regularity of the MPC solutions.
\begin{Assumption}
	\label{ass:terminal_cost}
	The terminal cost satisfies $V_\mathrm{f}(x)>-\infty$; $V_\mathrm{f}(x)< \infty$ for all bounded $x\in\mathcal{X}_\mathrm{f}$, which we assume to be closed and to satisfy $\mathcal{X}_\mathrm{f}\subseteq\mathcal{X}_h$; $V_\mathrm{f}(0)=0$; $V_\mathrm{f}$ is continuous on $\mathcal{B}_\delta\cap \mathcal{X}_\mathrm{f}$.
	For every compact set $\mathcal{\bar X} \subseteq \mathcal X_+$, there exists a finite integer $N_{\mathcal{\bar X}} \geq 0$ such that, for every $x \in \mathcal{\bar X}$, the trajectory $\{x_k^+\}_{k\geq 0}$ with $x_0^+=x$, $x_{k+1}^+=f(x_k^+,u_+(x_k^+))$ 
	satisfies $x_k^+ \in \mathcal X_{\mathrm f}$, $\forall\, k \geq N_{\mathcal{\bar X}}$. 
	Finally,
	\begin{align*}
		V_\mathrm{f}(x) \geq V_\ominus(x) + \eta(\|x\|),
	\end{align*}
	$\forall \, x\in\mathcal{X}_\mathrm{f}$, with $\eta\in\mathcal{PD}$, $\eta(\|x\|) \geq \beta \left(V_\oplus(x)-V_\ominus(x)\right)$, $\forall\,x\in\mathcal X_\mathrm{f}$, 
	for some function $\beta\in\mathcal K$ satisfying
	\begin{align}
		\beta(a) &\leq a-b+\beta(b), & \forall\,a>b\geq0. \label{eq:beta_condition}
	\end{align}
\end{Assumption}
The condition on $\beta$ requires (i) $\beta(a)\leq a$ for all $a>0$; and (ii) that $\beta$ is Lipschitz with constant $1$, i.e., it grows no faster than linearly. Consequently, if some function $\beta^0$ does not satisfy this condition, a new function $\beta$ satisfying it can be easily constructed by making it smaller for larger values of its argument. 
In this light, the assumption essentially requires that the terminal cost is larger than $V_\ominus$ by a sufficient margin close to the origin.

Under the assumptions of Theorem~\ref{thm:as_stab_N_long}, Lemma~\ref{lem:KL_stability} gives $\mathcal{KL}$-stability of the proper optimal closed loop $x_{k+1}^+=f(x_k^+,u_+(x_k^+))$. Consequently, the terminal entry requirement above holds whenever there exists $\delta_\mathrm{f}>0$ such that $\mathcal X_\mathrm{f}\supseteq \mathcal B_{\delta_\mathrm{f}}\cap\mathcal X_+$.

\begin{Assumption}
	\label{ass:compactness}
	For every $\xi>0$, there exist compact sets $\mathcal X^\xi$ and $\mathcal U^\xi$ such that, for every $N$ and for each of the terminal costs $V_\mathrm{f}(x)$ and $\hat V_\mathrm{f}(x)=V_\ominus(x)+\beta(V_\oplus(x)-V_\ominus(x))$, the corresponding OCP admits an optimal solution with $x_k^\star\in\mathcal X^\xi$ and $u_k^\star\in\mathcal U^\xi$ for every initial state in its feasible domain intersected with $\mathcal B_\xi$. Moreover, for $N$ sufficiently large and every feasible initial state, the optimal state trajectory of OCP~\eqref{eq:mpc} is unique. 
\end{Assumption}
This assumption guarantees that, uniformly with respect to the prediction horizon and bounded sets of initial conditions, at least one optimal solution can be selected whose state and input trajectories remain in common compact sets, while other optimal solutions may still admit arbitrarily large inputs. The first part of the assumption is automatically satisfied in case the constraint set $\mathcal H$ is compact, provided that the finite-horizon minima are attained, which is the case if, e.g., $V_\mathrm{f}$ is lower semicontinuous. The second part, similarly to what we assume in Assumption~\ref{ass:regularity}, requires uniqueness of the trajectory. This second assumption can possibly be relaxed at the expense of further technicalities.

\begin{Assumption}
	\label{ass:local_approx}
	In a neighborhood of the origin, $V_+$, $V_\ominus$, and $V_\mathrm{f}$ admit quadratic representations up to higher-order terms, with corresponding Hessian matrices $P_+$, $P_-$, and $P_\mathrm{f}$, satisfying $P_+-P_-\succ0$ and $P_\mathrm{f}\succ P_-$.
	Moreover, 
	MPC OCP~\eqref{eq:mpc} is locally first-order equivalent to its linear-quadratic approximation in the sense that, for all sufficiently large $N$, $u_N^\star(x) = u_N^{\mathrm{LQ}}(x) + \mathcal O(\|x\|^2)$, 
	where $u_N^{\mathrm{LQ}}$ denotes the receding-horizon feedback law associated with the linear-quadratic approximation of the OCP. The approximation above holds uniformly with respect to all sufficiently large $N$.
	Finally, the constraints active at the origin remain active along the local optimal predictions in a common neighborhood of the origin, uniformly with respect to all sufficiently large $N$. The linearization at the origin of the dynamics $A:=\nabla_x f(0,0)$, $B:=\nabla_u f(0,0)$ and active constraints yields a decomposition into a constrained subsystem and a free one, characterized by the pair $(A_\mathrm{f},B_\mathrm{f})$, which we assume to be stabilizable. 
	Moreover, $\hat P_+-\hat P_-\succ0$, $\hat P_\mathrm{f}-\hat P_-\succ0$, where $\hat P_+$, $\hat P_-$, and $\hat P_\mathrm{f}$ denote the Hessians at the origin of $V_+$, $V_\ominus$, and $V_\mathrm{f}$, respectively, expressed in the local feasible state coordinates associated with the free subsystem.
\end{Assumption}
This is a regularity assumption that allows us to simplify the discussion about the convergence of $V_N$ to $V_+$ in a neighborhood of the origin which allows us to establish asymptotic stability for a sufficiently long but finite horizon. 
Sufficient regularity conditions ensuring the local approximation above are discussed in~\cite{Faulwasser2018}. In particular, under suitable constraint qualifications and active-set regularity, the nonlinear optimal solution admits a first-order approximation given by the solution of the corresponding linear-quadratic OCP, with an error of order $\mathcal O(\|x\|^2)$. 

For the sake of simplicity, and in order to avoid introducing excessive technicalities related to controllability and the construction of recursively feasible regions, we assume next that, for all sufficiently long prediction horizons, there exists a nonempty positively invariant set for the MPC closed loop contained in $\mathcal X_+$. Conditions ensuring the existence of recursively feasible and asymptotically stabilizable subsets of the feasible region for sufficiently long prediction horizons are well studied in the MPC literature; see, e.g.,~\cite[Chapter~7]{Gruene2017} and the references therein. In particular, such sets can be constructed as suitable sublevel sets of the finite- or infinite-horizon value function under appropriate local controllability assumptions. We therefore state the existence of such a set directly, rather than imposing explicit controllability assumptions and reproducing the corresponding feasibility analysis.

\begin{Theorem}
	\label{thm:as_stab_N_long}
	Suppose that Assumptions~\ref{ass:regularity},~\ref{ass:relaxation}, \ref{ass:Bellman_regularity}, \ref{ass:compactness}, \ref{ass:terminal_cost}, and~\ref{ass:local_approx} hold. 
	Assume that either strict dissipativity or two-storage strict dissipativity holds for all bounded $x\in\mathcal{X}_h$. Assume that there exist $N_\mathrm{f}<\infty$ and a compact set $\mathcal F\neq\emptyset$ such that, for every $N>N_\mathrm{f}$, $\mathcal F\subseteq\mathcal X_N\cap\mathcal X_+$ and $f(x,u_N^\star(x))\in\mathcal F$ for all $x\in\mathcal F$. 
	Then, the value function $V_N$ defined in~\eqref{eq:mpc} converges to $V_+$ pointwise on $\mathcal X_h$, in the extended-real sense, and uniformly on every compact subset of $\mathcal X_+$ as $N\to\infty$. 
	Moreover, there exists $\bar N<\infty$ such that, for every $N\geq\bar N$, the feedback law $u_N^\star$ asymptotically stabilizes the origin for every $\hat x\in\mathcal F$.
\end{Theorem}
\begin{pf}
	The proof is split in two parts: we first prove that $V_N$ converges to $V_+$; then we prove that asymptotic stability is obtained for a sufficiently long but finite horizon.
	
	\underline{Part 1: $V_N$ converges to $V_+$.}
	Using value iteration, we define the sequence
	\begin{align*}
		V_{N+1}(x) &= \min_{u} \ \ell(x,u)  + H(x,u) + V_N(f(x,u)), \\
		V_0(x)&= V_\mathrm{f}(x),
	\end{align*}
	where  we define $V_\mathrm{f}(x)=\infty$ for all $x\notin\mathcal{X}_\mathrm{f}$. In order to prove that this sequence converges to $V_+$, we construct converging upper and lower bounds next.
	
	Let us begin with the upper bound. 
	Let $\{x_k^+\}_{k\geq 0}$ denote the optimal trajectory associated with $u_+$, i.e., $x_0^+=x$ and $x_{k+1}^+=f(x_k^+,u_+(x_k^+))$. Using $\lambda(x)=-V_\ominus(x)$ to obtain the rotated cost $L$ from stage cost $\ell$, we define the extended-real-valued function for $x\in\mathcal{X}_+$
	\begin{align}
		\bar\upsilon(x) := \sup_{N\geq 0} \ \sum_{k=0}^{N-1} L(x_k^+,u_+(x_k^+)) + \bar V_\mathrm{f}(x_N^+).
	\end{align}
	Since the term corresponding to $N=0$ is $\bar V_\mathrm{f}(x)$, we have $\bar\upsilon(x) \geq \bar V_\mathrm{f}(x)$. 
	Moreover, separating the term corresponding to $N=0$ from those with $N\geq 1$ yields
	\begin{align*}
		\bar\upsilon(x) 
		&= \max\Bigg\{ \bar V_\mathrm{f}(x), \ L(x,u_+(x)) + \bar\upsilon(f(x,u_+(x))) \Bigg\},
	\end{align*}
	where the equality is understood in the extended-real sense. Then,
	\begin{align}
		\label{eq:upsilon_decrease}
		\bar\upsilon(f(x,u_+(x))) + L(x,u_+(x)) \leq \bar\upsilon(x).
	\end{align}
	Furthermore, optimality of $u_+$ and $V_+$ for OCP~\eqref{eq:ocp_fwd} yields
	\begin{align*}
		\bar\upsilon(x) &= \bar V_+(x) + \sup_{N\geq 0} \left\{ \bar V_\mathrm{f}(x_N^+)-\bar V_+(x_N^+) \right\}.
	\end{align*}
	Fix $x\in\mathcal X_+$. By Assumption~\ref{ass:terminal_cost}, 
	there exists $N_x<\infty$ such that $x_N^+\in\mathcal X_\mathrm{f}$, $\forall\,N\geq N_x$. Since $x_N^+\to0$, continuity of $\bar V_\mathrm{f}$ and $\bar V_+$ at the origin gives $\bar V_\mathrm{f}(x_N^+)-\bar V_+(x_N^+) \to 0$. 
	Thus, for every $\bar N\geq N_x$, the tail sequence $\{\bar V_\mathrm{f}(x_N^+)-\bar V_+(x_N^+)\}_{N\geq\bar N}$ is real-valued and bounded, and
	\begin{align*}
		\bar\upsilon(x_{\bar N}^+) = \bar V_+(x_{\bar N}^+) +\sup_{N\geq\bar N}\bar V_\mathrm{f}(x_N^+)-\bar V_+(x_N^+) <\infty.
	\end{align*}
	Moreover, $\bar V_+(x_{\bar N}^+)\to0$, $\sup_{N\geq\bar N}\bar V_\mathrm{f}(x_N^+)-\bar V_+(x_N^+)\to0$, and, consequently, $\lim_{\bar N\to\infty} \bar\upsilon(x_{\bar N}^+)=0$. 
	Notice that $\bar\upsilon$ may take the value $+\infty$ at points whose trajectory has not yet entered $\mathcal X_\mathrm{f}$ permanently. In the auxiliary OCP introduced below, the effective domain of $\bar\upsilon$ therefore acts as an implicit terminal constraint.

	We can now define an upper bound for the rotated finite-horizon value function as
	\begin{subequations}
		\label{eq:mpc_tilde}
		\begin{align}
			\tilde{\bar{V}}_{N}(\hat x)
			:=
			\inf_{x,u}\ & \sum_{k=0}^{N-1}L(x_k,u_k) + \bar{\upsilon}(x_N) \hspace{-6em}&&\hspace{2em}\\
			\mathrm{s.t.}\ \, &	x_0=\hat x,	\\
			& x_{k+1}=f(x_k,u_k), && k\in\mathbb{I}_0^{N-1}, \\
			& h(x_k,u_k)\leq0,  && k\in\mathbb{I}_0^{N-1}, \\
			& x_N\in\mathcal{X}_+.
		\end{align}
	\end{subequations}
	Since $\bar \upsilon(x) \geq \bar V_\mathrm{f}(x)$, we have $\tilde{\bar{V}}_N(x)\geq\bar{V}_N(x)$. Moreover, $\bar \upsilon(x) \geq \bar V_+(x)$ entails $\tilde{\bar{V}}_N(x)\geq\bar{V}_+(x)$ for all $x\in\mathcal{X}_+$.
	
	To prove that the sequence $\tilde{\bar V}_N$ is nonincreasing, we observe that we can construct a feasible initial guess for the longer horizon problem by extending the trajectory using $u_+$ to obtain $x_{N+1}=f(x_N,u_+(x_N))$. The difference between the cost of this new trajectory and the previous one is nonpositive, as proven by~\eqref{eq:upsilon_decrease}, 
	such that $\tilde{\bar V}_{N+1}(x) \leq \tilde{\bar V}_N(x)$.
	Consequently, $\tilde{\bar V}_N(x)$ converges pointwise to some limit $\tilde{ \bar V}_\infty(x)\geq \bar V_\oplus(x)=\bar V_+(x)$ for all $x\in\mathcal{X}_+$. 
	
	To prove that in fact $\tilde{ \bar V}_\infty(x)=\bar V_+(x)$ for all $x\in\mathcal{X}_+$, let $\{x_k^+,u_k^+\}_{k\geq0}$ be an optimal trajectory for $V_+(x)$. Since this trajectory is feasible for the OCP defining $\tilde{\bar V}_N$; $\lim_{N\to\infty} x_N^+=0$ and $\bar \upsilon(0)=0$, we have
	\begin{align*}
		\lim_{N\to\infty}\tilde{\bar V}_N(x) = \tilde{\bar V}_\infty(x) = \bar V_+(x), \qquad x\in\mathcal X_+.
	\end{align*}
	
	In order to conclude the convergence proof, we construct a converging lower bound.
	Consider terminal cost 
	\begin{align*}
		\hat V_\mathrm{f}(x)&=\beta(V_\oplus(x)-V_\ominus(x))+V_\ominus(x),
	\end{align*}
	where $\beta$ is the function used in Assumption~\ref{ass:terminal_cost},
	yielding sequence $\hat V_N$ by value iteration. Note that $\hat V_\mathrm{f}$ is defined on $\mathcal X_h$, and extended by $+\infty$ outside 	$\mathcal X_h$. 
	Rotate the sequence using $\lambda(x)=-V_\ominus(x)$, to obtain $\hat{\bar{V}}_N(x) = \hat V_N(x) - V_\ominus(x)\geq0$ and $L(x,u)\geq0$. 
	Then, because $\bar V_\oplus(x)=V_\oplus(x)-V_\ominus(x)$, we have
	\begin{align*}
		\hat{\bar{V}}_1(x) &= L_1(x,\hat u_1^\star(x)) + \beta({\bar{V}}_\oplus(f(x,\hat u_1^\star(x)))) \geq0. 
	\end{align*}
	We observe that for $a\geq0$, $b\geq0$ we have
	\begin{align*}
		\beta(a+b) \leq a + b - b + \beta(b) = a + \beta(b),
	\end{align*}
	such that
	\begin{align*}
		\hat{\bar{V}}_1(x) &= L_1(x,\hat u_1^\star(x)) + \beta({\bar{V}}_\oplus(f(x,\hat u_1^\star(x)))) \\
		&\geq \beta\left ( L_1(x,\hat u_1^\star(x)) + {\bar{V}}_\oplus(f(x,\hat u_1^\star(x))) \right ) \\
		&\geq \beta\left ( \mathcal{L}_1(x,u_\oplus(x),z_\oplus(x)) + {\bar{V}}_\oplus(f(x,u_\oplus(x))+z_\oplus(x)) \right ) \\
		&=\beta\left ( \bar V_\oplus(x) \right ).
	\end{align*}
	Consequently, by Lemma~\ref{lem:value_iteration_bound}, $\hat V_{N+1}(x)\geq\hat V_N(x)$. Thus, the sequence $\{\hat V_N\}_{N\geq0}$ admits a pointwise limit, possibly in the extended-real sense, which we denote by $\hat V_\infty(x):=\lim_{N\to\infty}\hat V_N(x)$. 
	We prove next that $\hat V_\infty=V_+$.
	
	We first prove that $\hat{\bar V}_N^1$ is lower semicontinuous for every $N$, with value $+\infty$ at infeasible states. By Assumption~\ref{ass:relaxation}, $V_\oplus-V_\ominus$ is lower semicontinuous. Since $\beta\in\mathcal K$ is continuous and increasing, $\hat{\bar V}_0^1=\beta(V_\oplus-V_\ominus)$ is lower semicontinuous.
	
	Suppose that $\hat{\bar V}_{N-1}^1$ is lower semicontinuous, and let $x_j\to x$ in $\mathcal X_h$. To establish lower semicontinuity of $\hat{\bar V}_N^1$,
	we must prove $\hat{\bar V}_N^1(x) \leq \liminf_{j\to\infty} \hat{\bar V}_N^1(x_j)=:a$. 
	If $a=+\infty$, the required inequality is immediate. Otherwise, passing to a subsequence, we may assume that the values are finite and converge to $a$. Choose $\xi>0$ such that $x_j\in\mathcal B_\xi$. By Assumption~\ref{ass:compactness}, for every $j$ we can select an optimal trajectory whose first input $u_j$ belongs to the common compact set $\mathcal U^\xi$. Bellman optimality yields $\hat{\bar V}_N^1(x_j) =L_1(x_j,u_j) +\hat{\bar V}_{N-1}^1(f(x_j,u_j))$.
	Passing to a further subsequence, $u_j\to\bar u$. Closedness of $\mathcal H$ gives $(x,\bar u)\in\mathcal H$. Moreover, continuity of $f$ yields $f(x_j,u_j)\to f(x,\bar u)$. The induction hypothesis therefore implies $\hat{\bar V}_{N-1}^1(f(x,\bar u)) \leq \liminf_{j\to\infty} \hat{\bar V}_{N-1}^1(f(x_j,u_j))$. 
	Combining this inequality with continuity of $L_1$ and the Bellman equality above gives
	\begin{align*}
		\hat{\bar V}_N^1(x)
		&\leq L_1(x,\bar u) +\hat{\bar V}_{N-1}^1(f(x,\bar u)) \\
		&\leq \liminf_{j\to\infty} \hat{\bar V}_N^1(x_j)=a,
	\end{align*}
	Since the original sequence $x_j\to x$ was arbitrary, this proves lower semicontinuity of $\hat{\bar V}_N^1$. The induction is therefore complete. 
	Finally, $V_\ominus$ is continuous by Assumption~\ref{ass:Bellman_regularity}. Hence, $\hat V_N=\hat{\bar V}_N^1+V_\ominus$ is also lower semicontinuous for every $N$.
	
	Fix any $\xi>0$ and $x\in\operatorname{dom}(\hat V_\infty)\cap\mathcal B_\xi$. For every $N$, let $\{x_k^N,u_k^N\}_{k=0}^{N}$ be an optimal trajectory associated with $\hat V_N(x)$, selected according to Assumption~\ref{ass:compactness}, i.e., $x_k^N\in\mathcal X^\xi$ and $u_k^N\in\mathcal U^\xi$ independent of $N$. Hence, there exist a subsequence $\{N_j\}_{j\geq0}$ and a feasible infinite trajectory $\{x_k^\infty,u_k^\infty\}_{k\geq0}$ such that, for every fixed $k$, we have $x_k^{N_j} \to x_k^\infty\in\mathcal X^\xi$, and $u_k^{N_j} \to u_k^\infty\in\mathcal U^\xi$, as $j\to\infty$.
	
	We prove next that this trajectory is Bellman optimal for
	$\hat V_\infty$. For $M\geq1$ and all sufficiently large $j$
	\begin{align*}
		\hat V_{N_j}(x) = \sum_{k=0}^{M-1}\ell(x_k^{N_j},u_k^{N_j}) + \hat V_{N_j-M}(x_M^{N_j}).
	\end{align*}
	Since $\hat V_{N_j}(x)\to\hat V_\infty(x)$ and
	$(x_k^{N_j},u_k^{N_j})\to(x_k^\infty,u_k^\infty)$ for every fixed
	$k$, it follows that
	\begin{align}
		\lim_{j\to\infty}\hat V_{N_j-M}(x_M^{N_j}) = \hat V_\infty(x) - \sum_{k=0}^{M-1}\ell(x_k^\infty,u_k^\infty).
		\label{eq:Vhat_tail_limit}
	\end{align}
	Fix $\bar N\geq0$. For all sufficiently large $j$,
	$N_j-M\geq\bar N$, and 
	$\hat V_{N_j-M}(x_M^{N_j}) \geq \hat V_{\bar N}(x_M^{N_j})$ yields
	\begin{align*}
		\lim_{j\to\infty}\hat V_{N_j-M}(x_M^{N_j}) &\geq \liminf_{j\to\infty}\hat V_{\bar N}(x_M^{N_j}) \geq \hat V_{\bar N}(x_M^\infty),
	\end{align*}
	where, for the last inequality, we exploited lower semicontinuity of $\hat V_{\bar N}$ and $x_{M}^{N_j}\to x_{M}^\infty$. 
	Since this holds for every $\bar N$, taking the supremum over
	$\bar N$ yields $\lim_{j\to\infty}\hat V_{N_j-M}(x_M^{N_j}) \geq \hat V_\infty(x_M^\infty)$. 
	Combining this inequality with~\eqref{eq:Vhat_tail_limit} gives
	\begin{align*}
		\hat V_\infty(x) \geq \sum_{k=0}^{M-1} \ell(x_k^\infty,u_k^\infty) + \hat V_\infty(x_M^\infty).
	\end{align*}
	The reverse inequality follows by observing that the limit for $N\to \infty$ of $\hat V_N(x) \leq \ell(x,u) + \hat V_{N-1}(f(x,u))$ yields
	\begin{align*}
		\hat V_\infty(x) \leq \ell(x,u) + \hat V_\infty(f(x,u)).
	\end{align*}
	Iterating this inequality along the same feasible trajectory gives the reverse inequality.
	Therefore,
	\begin{align}
		\hat V_\infty(x) &= \sum_{i=0}^{M-1}\ell(x_i^\infty,u_i^\infty) + \hat V_\infty(x_M^\infty), && \forall\,M\geq0.
		\label{eq:Vhat_prefix_eq}
	\end{align}
	Subtracting~\eqref{eq:Vhat_prefix_eq} with $M=k$ from~\eqref{eq:Vhat_prefix_eq} with $M=k+1$ yields $\hat V_\infty(x_k^\infty) = \ell(x_k^\infty,u_k^\infty) + \hat V_\infty(x_{k+1}^\infty)$, $\forall\,k\geq0$.
	Thus, $\{x_k^\infty,u_k^\infty\}_{k\geq0}$ is a bounded
	optimal trajectory associated with $\hat V_\infty$.
	
	Finally, since $\hat{\bar V}^{\,1}_\infty(x) \geq \beta(\gamma(\|x\|))$, for all $x\in\operatorname{dom}(\hat V_\infty)$, Lemma~\ref{lem:V_dagger} yields
	\begin{align}
		\operatorname{dom}(\hat V_\infty) \subseteq \mathcal X_+, && \hat V_\infty(x)\geq V_+(x), && \forall\,x\in\operatorname{dom}(\hat V_\infty).
		\label{eq:hatVinf_subset_Xplus}
	\end{align}
	We prove next the reverse inclusion. For every $x\in\mathcal X_+$, the upper and lower value-function sequences satisfy $\hat V_N(x)\leq\tilde V_N(x)$,  $\forall N\geq0$. Moreover, $\tilde V_N(x)$ is nonincreasing and $\tilde V_N(x)\to V_+(x)<\infty$. Hence, there exists $N_x<\infty$ such that
	$\tilde V_{N_x}(x)<\infty$. Consequently, for every $N\geq N_x$, $\hat V_N(x) \leq \tilde V_N(x) \leq \tilde V_{N_x}(x) <\infty$. 
	Since $\hat V_N(x)\uparrow\hat V_\infty(x)$ and the sequence is bounded from above by $\tilde V_{N_x}(x)$, it follows that $\hat V_\infty(x)<\infty$. Therefore, $\mathcal X_+ \subseteq \operatorname{dom}(\hat V_\infty)$ and $\hat V_\infty(x)\leq V_+(x)$, $\forall\,x\in\mathcal{X}_+$. 
	Together with~\eqref{eq:hatVinf_subset_Xplus}, this proves
	\begin{align*}
		\operatorname{dom}(\hat V_\infty) = \mathcal X_+, \qquad \hat V_\infty(x)=V_+(x), \quad \forall\,x\in\mathcal X_+.
	\end{align*}
	Consequently, lower-semicontinuity yields, for every sequence $N_j\to\infty$ and $x_j\to x$,
	\begin{align}
		\liminf_{j\to\infty}\hat V_{N_j}(x_j) \geq \hat V_\infty(x) = V_+(x),
		\label{eq:hatV_liminf}
	\end{align}
	where the equality is understood in the extended-real sense, such that $x\in\mathcal X_h\setminus\mathcal X_+ \implies \hat V_N(x)\to+\infty$.
	
	In order to conclude the proof of the first claim, we observe that $\bar V_\mathrm{f}(x)\geq \eta(\|x\|)$, and Assumption~\ref{ass:terminal_cost} guarantees that $\eta(\|x\|) \geq \beta(\bar V_\oplus(x))$, 
	such that $V_N(x)\geq \hat V_N(x)$. 
	Consequently, the sequence $V_N$ is bounded from above by $\tilde V_N$ and from below by $\hat V_N$, both of which converge to $V_+$ as $N\to\infty$. Therefore, also $V_N$ must converge to $V_+$ as $N\to\infty$. 
	
\underline{Part 2: asymptotic stability for finite $N$.} 
	We can now prove that, for a sufficiently long prediction horizon $N$, the feedback law $u_N^\star(x)$ is asymptotically stabilizing. 
	We first show that the Lyapunov function $\mathcal V_+$ used in~\eqref{eq:L3} can be chosen to be continuous. 
	By Assumptions~\ref{ass:regularity} and~\ref{ass:Bellman_regularity}, $f(x,u_+(x))$ is single-valued, locally bounded, and has a closed graph; hence it is continuous on $\mathcal X_+$.
	
	By Lemma~\ref{lem:KL_stability}, there exists
	$\beta_+\in\mathcal{KL}$ such that, defining $x_{k+1}^+:=f(x_k^+,u_+(x_k^+))$ with $x_0^+:=x$, we have $\|x_k^+\|\leq\beta_+(\|x\|,k)$. Using the comparison-function construction in~\cite[Theorem~6.2]{Gruene2014a}, exploiting Sontag's lemma~\cite{Sontag1998a} one can choose $\alpha_1,\alpha_2\in\mathcal K_\infty$ such that
	\begin{align*}
	\alpha_1(\beta_+(r,k))\leq \alpha_2(r)2^{-k}, \qquad r\geq0,\quad k\geq0.
	\end{align*}
	Define
	\begin{align*}
	\mathcal V_+(x):= \sum_{k=0}^{\infty}\alpha_1(\|x_k^+\|).
	\end{align*}
	For every $R<\infty$, the summands are continuous and the series converges uniformly on $\mathcal X_+\cap\mathcal B_R$, since its tail from index $m$ is bounded by $2\alpha_2(R)2^{-m}$. Consequently, $\mathcal V_+$ is continuous on $\mathcal X_+$. Moreover,
	\begin{align*}
		\alpha_1(\|x\|)\leq\mathcal V_+(x)&\leq2\alpha_2(\|x\|),\\
		\mathcal V_+(f(x,u_+(x)))-\mathcal V_+(x)&=-\alpha_1(\|x\|).
	\end{align*}
	
	Using the rotated cost defined in~\eqref{eq:L3} with $\lambda_1(x)=-V_\ominus(x)$, we obtain the rotated value function $\bar V_N^3$, which satisfies, for $\bar\beta\in\mathcal{K}$, $\bar\gamma\in\mathcal{K}$,
	\begin{align*}
		\bar \beta(\|x\|)\geq\bar V_N^3(x) \geq \bar \gamma(\|x\|).
	\end{align*}
	The lower bound follows from $\bar V_N^3(x)=\bar V_N^1(x) + \mathcal{V}_+(x)\geq \bar \gamma(\|x\|)$, while the upper bound $\bar \beta(\|x\|)\geq\bar V_N^3(x)$ is again a consequence of Proposition~\ref{prop:bound}. 
	Indeed, for $N$ large enough such that the infinite-horizon optimal trajectory $\{x_k^+\}$ is a feasible candidate, we have
	\begin{align*}
		V_\ominus(x)\leq V_N(x) \leq V_+(x)-V_+(x_N^+)+V_f(x_N^+),
	\end{align*}
	and both bounds converge to zero as $x\to0$, such that $V_N$ is continuous at the origin. The comparison function $\bar\beta$ can be selected independently of all sufficiently large $N$, since the value functions converge uniformly on $\mathcal F$, are uniformly bounded there, and are uniformly continuous at the origin.
	
	The stability condition then becomes
	\begin{align}
		\label{eq:Lyap_decrease}
		\bar V_N^3(f(x,u_N^\star(x))) - \bar V_N^3(x) \leq  -\sigma(\|x\|),
	\end{align}
	for some $\sigma\in\mathcal{K}$, for all $x\in\mathcal{F}$. 
	We know that
	\begin{align*}
		\bar V_N^3(x) = L_3(x,u_N^\star(x)) + \bar V_{N-1}^3(f(x,u_N^\star(x))),
	\end{align*}
	such that 
	\begin{align*}
		&\bar V_N^3(f(x,u_N^\star(x))) - \bar V_N^3(x) \\
		&\hspace{5em}= -L_3(x,u_N^\star(x)) \\
		&\hspace{6em} + \bar V_N^3(f(x,u_N^\star(x))) - \bar V_{N-1}^3(f(x,u_N^\star(x))),
	\end{align*}
	and the decrease condition becomes
	\begin{align*}
		&-\sigma(\|x\|)+L_3(x,u_N^\star(x)) \geq \\
		&\hspace{6em}\bar V_N^3(f(x,u_N^\star(x))) - \bar V_{N-1}^3(f(x,u_N^\star(x))),
	\end{align*}
	where we stress that, by assumption, $x\in\mathcal{F}$ entails $f(x,u_N^\star(x))\in\mathcal{F}\subseteq\mathcal{X}_+$. 
	In order to prove the decrease condition, we first prove that the convergence of $\tilde V_N$ and $\hat V_N$ is uniform on compact sets. 
	
	Let $\mathcal{\bar X}\subseteq\mathcal X_+$ be compact. Since the sequence $\tilde{\bar V}_N$ is nonincreasing and converges pointwise to $\bar V_+$, we have $\tilde{\bar V}_N(x)\geq\bar V_+(x)$. Moreover, the optimal trajectory $\{x_k^+(x),u_k^+(x)\}_{k\geq0}$ of~\eqref{eq:ocp_fwd} is feasible for OCP~\eqref{eq:mpc_tilde} defining $\tilde{\bar V}_N(x)$, such that
	\begin{align*}
		\tilde{\bar V}_N(x)-\bar V_+(x) &\leq \bar\upsilon(x_N^+(x)) - \bar V_+(x_N^+(x)) \\
		&= \sup_{M\geq N} \left\{ \bar V_\mathrm{f}(x_M^+(x)) - \bar V_+(x_M^+(x)) \right\}.
	\end{align*}
	By $\mathcal{KL}$-stability of the optimal closed loop we have $\sup_{x\in\mathcal{\bar X}} \sup_{M\geq N} \|x_M^+(x)\| \to 0$, as $N\to\infty.$
	Since both $\bar V_\mathrm{f}$ and $\bar V_+$ are continuous at the origin and vanish at the origin, 
	$\sup_{x\in\mathcal{\bar X}} \sup_{M\geq N} \left| \bar V_\mathrm{f}(x_M^+(x)) - \bar V_+(x_M^+(x)) \right| \to 0,$
	and, in turn, $\sup_{x\in\mathcal{\bar X}} \left| \tilde{\bar V}_N(x)-\bar V_+(x) \right| \to0$, 
	i.e., $\tilde{\bar V}_N\to\bar V_+$ uniformly on compact subsets of $\mathcal X_+$.
	
	Since $\hat V_N(x)\uparrow V_+(x)$, $\hat V_N$ is lower semicontinuous, and $V_+$ is continuous, 
	Dini's theorem for semicontinuous functions implies that the convergence is uniform on compact sets. Moreover, since $\hat V_N(x)\leq V_N(x)\leq\tilde V_N(x)$, it follows that $V_N\to V_+$	uniformly on compact subsets of $\mathcal X_+$. Since the rotation does not depend on $N$, the same conclusion holds for the rotated value functions, i.e., $\bar V_N^3\to\bar V_+^3$	uniformly on compact sets. In turn, $\bar V_N^3-\bar V_{N-1}^3\to0$ uniformly on compact sets.
	
	We prove next that the successor state generated by the finite-horizon optimizer approaches the infinite-horizon optimal successor uniformly on $\mathcal F$. Consider arbitrary sequences $N_j\to\infty$ and $x_j\in\mathcal F$. Since $\mathcal F$ is compact, upon extraction of a subsequence we have $x_j \to x\in\mathcal F$ and $f(x_j,u_{N_j}^\star(x_j))\to\bar x^+\in\mathcal F$. 
	Moreover, by Assumption 4.6, optimal inputs $u_j\in u^\star_{N_j}(x_j)$ can be selected in a compact set independent of $j$. By uniqueness of the optimal state trajectory, these inputs yield the same successor state as any other optimal input. Hence, possibly extracting a further subsequence, 
	$\exists \, u_j\in u_{N_j}^\star(x_j)$ and $\bar u$ such that $u_j\to\bar u$, $f(x_j,u_j)\to f(x,\bar u)=\bar x^+$.
	
	Since $V_{N_j}(x_j) = \ell(x_j,u_{j}) + V_{N_j-1}(f(x_j,u_{j}))$, and $V_N\to V_+$ uniformly on $\mathcal{F}$, we have $V_{N_j}(x_j)\to V_+(x)$. Moreover, since $V_N\geq\hat V_N$, Equation~\eqref{eq:hatV_liminf} gives
	\begin{align*}
		\liminf_{j\to\infty}V_{N_j-1}(f(x_j,u_{j})) 
		&\geq V_+(f(x,\bar u)).
	\end{align*}
	Therefore, by continuity of $\ell$,
	\begin{align*}
		V_+(x) &= \liminf_{j\to\infty} \left[ \ell(x_j,u_{j}) + V_{N_j-1}(f(x_j,u_{j})) \right] \\ 
		&\geq \ell(x,\bar u)+V_+(f(x,\bar u)).
	\end{align*}
	On the other hand, since $(x,\bar u)\in\mathcal H$, the Bellman inequality~\eqref{eq:bellman_inequality} yields the reverse inequality. 
	Hence, $V_+(x) = \ell(x,\bar u)+V_+(f(x,\bar u))$. 
	In particular,
	$V_+(f(x,\bar u))<\infty$, so that $f(x,\bar u)\in\mathcal X_+$ is an
	optimal successor for the infinite-horizon problem~\eqref{eq:ocp_fwd}.
	By uniqueness of the optimal state trajectory, $f(x,\bar u)=f(x,u_+(x))$.
	
	Since the sequences $N_j\to\infty$ and $x_j\in\mathcal{F}$ were arbitrary,
	this further proves
	\begin{align*}
		\sup_{x\in\mathcal{F}} \left\|f(x,u_N^\star(x))-f(x,u_+(x))\right\| &\to0, && N\to\infty.
	\end{align*}
	Otherwise, there would exist $\varepsilon>0$, $N_j\to\infty$, and $x_j\in\mathcal{F}$ such that $\left\| f(x_j,u_{N_j}^\star(x_j))-f(x_j,u_+(x_j)) \right\| \geq\varepsilon$. 
	Upon extraction of a subsequence, $x_j\to x\in\mathcal{F}$, and $f(x_j,u_{N_j}^\star(x_j)) \to f(x,u_+(x))$. 
	By compactness of the infinite-horizon optimal trajectories, a further
	subsequence of $f(x_j,u_+(x_j))$ converges. Closedness of the optimal
	successor relation and uniqueness of the optimal state trajectory imply
	that its limit must also be $f(x,u_+(x))$, which contradicts the
	inequality above.
	
	Consequently, from
	\begin{align*}
		L_3(x,u_N^\star(x)) = \bar V_N^3(x) - \bar V_{N-1}^3(f(x,u_N^\star(x))),
	\end{align*}
	uniform convergence of the value functions and continuity of $\bar V_+^3$ imply
	\begin{align*}
		L_3(x,u_N^\star(x)) \to \bar V_+^3(x)-\bar V_+^3(f(x,u_+(x))) = L_3(x,u_+(x))
	\end{align*}
	uniformly on $\mathcal{F}$. 
	
	We can now go back to the decrease condition
	\begin{align*}
		&-\sigma(\|x\|)+L_3(x,u_N^\star(x)) \geq \\
		&\hspace{6em}\bar V_N^3(f(x,u_N^\star(x))) - \bar V_{N-1}^3(f(x,u_N^\star(x))).
	\end{align*}
	In the first term, we have that $L_3(x,u_N^\star(x))$ converges uniformly to $L_3(x,u_+(x)) \geq \alpha_3(\|x\|)$, 
	while the right hand side converges uniformly to zero. 
	Thus, for every $\varepsilon>0$, the required strict decrease holds
	uniformly on $\mathcal F\setminus\operatorname{int}\mathcal B_\varepsilon$
	for all sufficiently large $N$. 
	
	It remains to consider an arbitrarily small neighborhood of the origin. By Assumption~\ref{ass:local_approx}, the persistent active constraints can be treated as equalities in the local LQ approximation. We use the corresponding decomposition and work with the free subsystem $(A_\mathrm{f},B_\mathrm{f})$, which is stabilizable by assumption.
	
	The forward Bellman inequality and \eqref{eq:V_minus_optimality_relaxed_inequality} imply that $-V_+$ and $-V_\ominus$ are storage functions for the original dynamics. Taking second variations along the active-constraint manifold yields two quadratic storage inequalities for the reduced LQ problem. By Assumption~\ref{ass:local_approx}, the corresponding reduced Hessians satisfy
	\begin{align*}
		\hat P_+-\hat P_- &\succ0, &
		\hat P_\mathrm{f}-\hat P_- &\succ0.
	\end{align*}
	Thus, \cite[Theorem~4.10]{Zanon2025a} yields quadratic strict pre-dissipativity of the reduced problem. On its controllable subspace, \cite[Lemma~4.12]{Zanon2025a} bounds the stabilizing reverse-CGDARE solution from above by $\widehat P_-$. Consequently, the terminal-cost condition of \cite[Theorem~5.2]{Zanon2025a} holds. Together with its stabilizable extension, this yields exponential stability of the reduced LQ receding-horizon closed loop for all sufficiently large $N$.
	
	Denote the corresponding linear-quadratic feedback law as $u_N^{\mathrm{LQ}}$. By Assumption~\ref{ass:local_approx},
	\begin{align}
		u_N^\star(x) &= u_N^{\mathrm{LQ}}(x) + \mathcal O(\|x\|^2), && x \to 0,
	\end{align}
	uniformly with respect to all sufficiently large $N$. 
	We now work in local feasible state coordinates $x_\mathrm{f}$ associated with the free subsystem. The uniform first-order approximation above and twice continuous differentiability of the dynamics give
	\begin{align*}
		\hat f_{\mathrm{f},N}(x_\mathrm{f})
		&=
		A_{\mathrm{f},N}^{\mathrm{cl}}x_\mathrm{f}
		+\mathcal O(\|x_\mathrm{f}\|^2),
	\end{align*}
	uniformly for all sufficiently large $N$, where $\hat f_{\mathrm{f},N}$ denotes the nonlinear closed-loop map in these coordinates and $A_{\mathrm{f},N}^{\mathrm{cl}}$ is the reduced LQ closed-loop matrix. 
	The convergence of the Riccati iterations established in~\cite[Theorem~4.9]{Zanon2025}, together with~\cite[Theorem~5.2]{Zanon2025a}, implies convergence of $A_{\mathrm f,N}^{\mathrm{cl}}$ to a Schur matrix. Consequently, a common quadratic Lyapunov function provides a uniform strict stability margin for all sufficiently large $N$. 
	Hence, the higher-order perturbation can be absorbed in a sufficiently small neighborhood of the origin, uniformly with respect to all sufficiently large $N$. Consequently, the nonlinear receding-horizon closed loop is locally exponentially stable for all sufficiently large prediction horizons.
	
	Finally, we connect the strict decrease obtained away from the origin with the local stability result above. Let $\varepsilon>0$ be sufficiently small such that $\mathcal B_\varepsilon$ is contained in the neighborhood in which the local exponential stability result holds. Since $\mathcal F$ is compact, $\mathcal F_\varepsilon := \mathcal F\setminus\operatorname{int}\mathcal B_\varepsilon$	is compact. If $\mathcal F_\varepsilon=\emptyset$, the claim follows directly from local exponential stability. Otherwise, since $\alpha_3$ is positive definite, we have $c_\varepsilon := \min_{x\in\mathcal F_\varepsilon}\alpha_3(\|x\|) >0$, and $L_3(x,u_+(x)) \geq \alpha_3(\|x\|) \geq c_\varepsilon$ for all $x\in\mathcal F_\varepsilon$.
	
	By the uniform convergence established above,
	\begin{align*}
		\sup_{x\in\mathcal F} \left| L_3(x,u_N^\star(x))-L_3(x,u_+(x)) \right| &\to0, && N\to\infty.
	\end{align*}
	Therefore, for all sufficiently large $N$,
	\begin{align*}
		L_3(x,u_N^\star(x)) &\geq L_3(x,u_+(x))-\frac{1}{2}c_\varepsilon \geq \frac{1}{2}c_\varepsilon, && x\in\mathcal F_\varepsilon.
	\end{align*}
	Furthermore, since $\bar V_N^3-\bar V_{N-1}^3\to0$ uniformly on the compact set $\mathcal F$, for all sufficiently large $N$,
	\begin{align*}
		\left| \bar V_N^3(x)-\bar V_{N-1}^3(x) \right| &\leq \frac{1}{4}c_\varepsilon, && x\in\mathcal F.
	\end{align*}
	Since $\mathcal F$ is positively invariant under $u_N^\star$, for every $x\in\mathcal F_\varepsilon$ we have
	$f(x,u_N^\star(x))\in\mathcal F$. Consequently,
	\begin{align*}
		&\bar V_N^3(f(x,u_N^\star(x)))-\bar V_N^3(x)
		\\
		&\hspace{3em} = -L_3(x,u_N^\star(x)) \\
		&\hspace{5em} +\bar V_N^3(f(x,u_N^\star(x))) -\bar V_{N-1}^3(f(x,u_N^\star(x))) \\
		&\hspace{3em} \leq -\frac{1}{4}c_\varepsilon <0.
	\end{align*}
	Hence a closed-loop trajectory starting in $\mathcal F$ cannot remain in $\mathcal F_\varepsilon$ indefinitely, since $\bar V_N^3$ is bounded from below on $\mathcal F$. It therefore enters $\mathcal B_\varepsilon$ in finite time. Once inside this neighborhood, the local exponential stability result established above implies convergence to the origin. Since $\mathcal F$ is positively invariant, this proves asymptotic stability for every initial condition $\hat x\in\mathcal F$.
	$\hfill\qed$
\end{pf}

\section{Relationship With The Cost-To-Travel}
\label{sec:cost-to-travel}

In this section we discuss the similarities and differences of the approach we propose with the ideas developed in~\cite{Houska2015,Houska2017}, as the two approaches are based on the same concept, though the slight differences in the formulation result in rather different analysis and results.

We proved that storage functions and value functions are tightly related. However, this link does not generally extend to fixed finite horizons. One might then ask whether it is possible to relate storage functions also to finite-horizon value functions. Unfortunately, if we define the finite-horizon problems similarly to~\eqref{eq:ocp_fwd}-\eqref{eq:ocp_bwd_relaxed} with a prediction horizon truncated to $\pm N$, the presence of a terminal/initial point constraint of the form $x_{\pm N}=0$, the answer is generally negative, as 
\begin{align*}
	V_{+,N}(x) &\geq V_+(x), & V_{-,N}(x) &\leq V_-(x), \\
	V_{\oplus,N}(x) &\geq V_\oplus(x), & V_{\ominus,N}(x) &\leq V_\ominus(x), 
\end{align*}
such that, by Lemma~\ref{lem:diss_implies_boundedness}, $\lambda(x)=-V_{\dagger,N}(x)$, $\dagger\in\{+,-,\oplus,\ominus\}$, can satisfy dissipativity on $\mathcal X_h$ only if the corresponding inequality holds with equality on $\mathcal X_+$ for $\dagger\in\{+,\oplus\}$, or on $\mathcal X_-$ for $\dagger\in\{-,\ominus\}$.

Nevertheless, a similar idea has been used in~\cite{Houska2015} to prove asymptotic stability for the periodic case in continuous time, where, however, the horizon length of both the MPC problem and the value function used to rotate the cost are optimized. This optimization is in fact the key idea that makes it possible to prove the desired result.
A key concept is the so-called cost-to-travel function, introduced under this name later in~\cite{Houska2017}, and defined as
\begin{subequations}
	\label{eq:cost-to-travel}
	\begin{align}
		C(a,b,N) = \min_{x,u} \ & \sum_{k=0}^{N-1} \ell(x_k , u_k) \\
		\mathrm{s.t.} \ & x_0 = a, \\
		&x_{k+1} = f(x_k,u_k),  && k\in\mathbb{I}_0^{N-1}, \\
		&h(x_k,u_k) \leq 0, && k\in\mathbb{I}_0^{N-1},  \\
		&x_N = b,
	\end{align}
\end{subequations}
or its relaxation based on the choice $p(x,u,z)=r\|z\|_1$
\begin{subequations}
	\label{eq:cost-to-travel_relaxed}
	\begin{align}
		C_\mathrm{r}(a,b,N) = \min_{x,u,z} \ & \sum_{k=0}^{N-1} \ell(x_k , u_k) + r\|z_k\|_1 \\
		\mathrm{s.t.} \ & x_0 = a, \\
		&x_{k+1} = f(x_k,u_k)+z_k,\hspace{-7pt} && k\in\mathbb{I}_0^{N-1},  \\
		&h(x_k,u_k) \leq 0, && k\in\mathbb{I}_0^{N-1},  \\
		&x_N = b.
	\end{align}
\end{subequations}
Note that, using these definitions, we have
\begin{align*}
	V_+(x) &=\phantom{-}C(x,0,\infty), & V_{+,N}(x) &= \phantom{-}C(x,0,N), \\
	V_-(x) &= -C(0,x,\infty), & V_{-,N}(x) &= -C(0,x,N), \\
	V_\oplus(x) &= \phantom{-}C_\mathrm{r}(x,0,\infty), & V_{\oplus,N}(x) &= \phantom{-}C_\mathrm{r}(x,0,N), \\
	V_\ominus(x) &= -C_\mathrm{r}(0,x,\infty), &V_{\ominus,N}(x) &= -C_\mathrm{r}(0,x,N),
\end{align*}
where, for $N=\infty$, we obtain the infinite-horizon problems~\eqref{eq:ocp_fwd}-\eqref{eq:ocp_bwd_relaxed} instead of the limits of finite-horizon exact-endpoint ones.

The main idea in~\cite{Houska2015}, adapted to the discrete-time setting for optimal steady-state operation is to solve 
\begin{align*}
	\min_{M\in\mathbb{I}_1^{N}} \ C(x,0,N-M) + C_\mathrm{r}(0,x,M).
\end{align*}
Note that, if $M$ were not minimized, the resulting control law would be the same as the one obtained without the so-called return constraint, enforced by the term $C_\mathrm{r}(0,x,M)=-V_{\ominus,M}(x)$.
We observe that optimality of $V_{\ominus,M+1}$ yields
\begin{align*}
	V_{\ominus,M+1}(f(x,u)) + \ell(x,u) &\geq V_{\ominus,M}(x),
\end{align*}
however, 
because $V_{\ominus,M+1}(x) \geq V_{\ominus,M}(x)$, we have 
\begin{align*}
	V_{\ominus,M+1}(x) &\geq V_{\ominus,M}(x)\leq V_{\ominus,M+1}(f(x,u)) + \ell(x,u).
\end{align*}
which does not allow us to conclude that $-V_{\ominus,M+1}$ can be a storage function satisfying dissipativity. Indeed, there exist cases in which
$-V_{\ominus,M+1}$ does not satisfy dissipativity.

In~\cite{Houska2015} the key idea is to optimize the prediction horizon. That allows one to construct a feasible initial guess for the next time by having a longer horizon for $C_\mathrm{r}$ and a shorter one for $C$, i.e.,
\begin{align*}
	&C(x_{t+1},0,N-M_t-1) + \Big (C(x_t,x_{t+1},1) + C_\mathrm{r}(0,x_t,M_t)\Big )\\
	&\geq\min_{M\in\mathbb{I}_1^{N}} \ C(x_{t+1},0,N-M) + C_\mathrm{r}(0,x_{t+1},M),
\end{align*}
where $x_{t+1}$ is the state obtained by applying the optimal feedback law from the current state $x_t$, and $M_t$ is the optimal prediction horizon splitting time computed at time $t$. The first term in the left hand side is an initial guess for the first term in the right hand side, while the other two terms in the left hand side, grouped in parentheses, are an initial guess for the second term in the right hand side. By optimality one immediately obtains the desired inequality.
This approach exploits the inequality 
\begin{align*}
	\ell(x,u) - V_{N,\ominus}(x) + V_{N+1,\ominus}(f(x,u)) \geq 0.
\end{align*}
However, this inequality cannot be directly used to generate a rotation of the cost in a standard dissipativity framework. Because it is well-known that
\begin{align*}
	&\sum_{k=0}^{N-1} L(x_k,u_k) + \bar V_\mathrm{f}(x_N) \\
	&\hspace{6em}= \lambda(x_0) + \sum_{k=0}^{N-1} \ell(x_k,u_k) + V_\mathrm{f}(x_N),
\end{align*}
one could be tempted to interpret the introduction of the return constraint through $C_\mathrm{r}(0,x,M)$ as a rotation using storage function $\lambda(x)=C_\mathrm{r}(0,x,M)$. 
However, as $M$ is optimized at every time instant, this interpretation yields the generally time-dependent family $\lambda_t(x)=C_\mathrm{r}(0,x,M_t)$. Consequently, this mechanism does not directly establish standard dissipativity with a single time-independent storage function.

\section{Examples}
\label{sec:examples}
In this section we provide two examples that highlight the theoretical results of this paper.
\begin{figure}[t]
	\includegraphics[width=\linewidth,clip,trim=0 0 0 0]{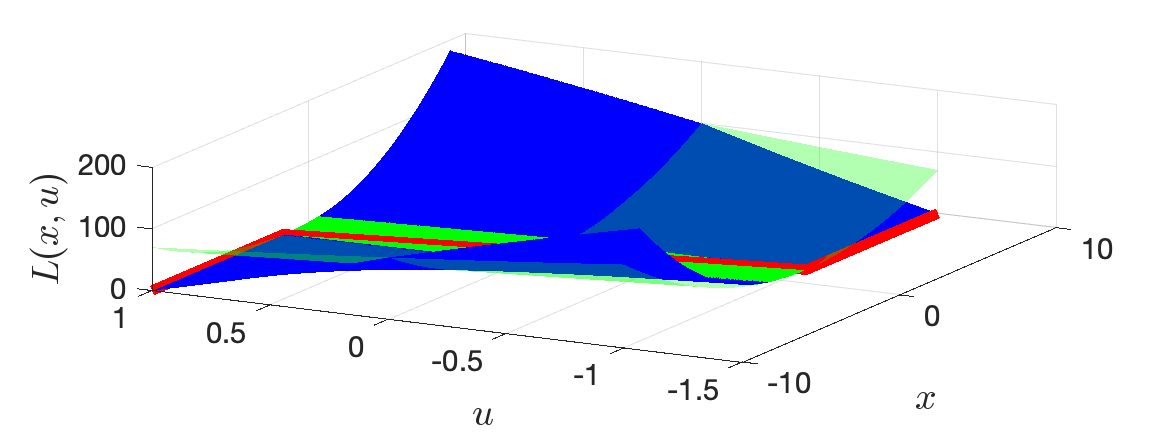}
	\includegraphics[width=\linewidth,clip,trim=0 0 0 0]{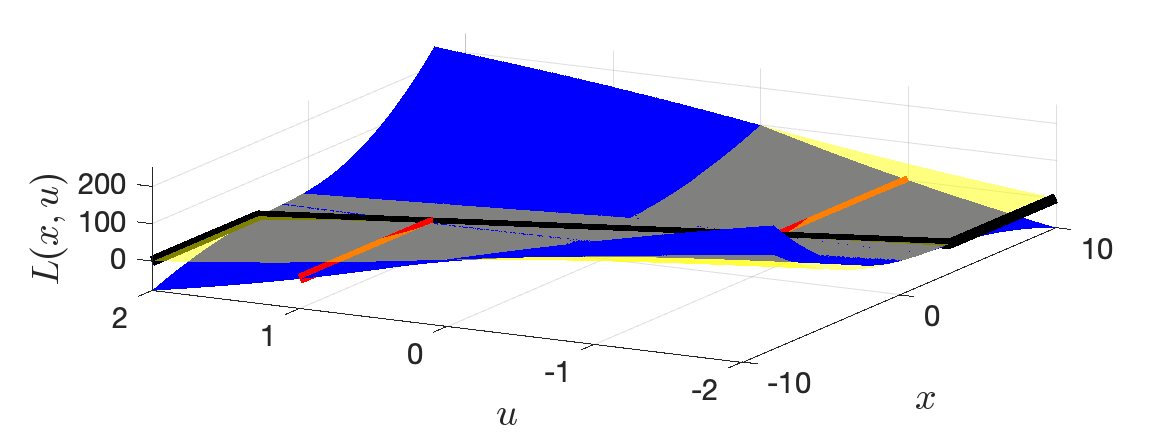}
	\caption{Constrained linear quadratic example: cost rotated using different storage functions and optimal feedback law.
	}
	\label{fig:mpc_bounds}
\end{figure}

\subsection{Constrained Linear Quadratic MPC}
\label{sec:lq_example}
We consider the simple linear system
\begin{align*}
	x_+&=x+u, & \ell(x,u) &= x^2 + u^2, \\
	x&\in[-10,10], & u&\in[-1,1].
\end{align*}
As we will also consider relaxed input constraints $u\in[-2,2]$, we will denote the corresponding value functions as $V_+^{[-1,1]}$ and $V_+^{[-2,2]}$ respectively. All value functions were computed using the MPT3 toolbox~\cite{Herceg2013}.

In Figure~\ref{fig:mpc_bounds}, we display the cost rotated using $-V_+^\mathrm{u}$ (i.e., the LQR value function) in light green and the cost rotated using $-V_+^{[-1,1]}$ in blue, together with the optimal feedback law as a red line in the top plot. In the bottom plot, we display the cost rotated using $-V_+^{[-1,1]}$ in blue together with the corresponding optimal feedback law as a red line; and the cost rotated using $-V_+^{[-2,2]}$ in light yellow with the corresponding optimal feedback as a black line. One can observe that the cost rotated using $-V_+^{[-1,1]}$ is valid for $x\in[-10,10]$ and $u\in[-1,1]$. The fact that this rotated cost is $0$ for $u_+(x)=\pm1$ suggests that selecting a control which is larger in absolute value for those states might yield a negative cost. Indeed, this rotated cost becomes invalid for $u\in[-2,2]$, as in the bottom plot one clearly sees that this cost becomes negative for $u\notin[-1,1]$ for some $x$. On the contrary, the cost rotated using $-V_+^{[-2,2]}$ is valid for all $u\in[-2,2]$, and $V_+^{[-2,2]}(x)\leq V_+^{[-1,1]}(x)$. Finally, as expected, the cost rotated using $-V_+^\mathrm{u}$ instead, is valid for all $x,u$. 
However, as we display using the solid green in the top plot, the value functions coincide wherever the unconstrained LQR trajectory is feasible. The corresponding rotated costs coincide for state-input pairs for which both the state and its successor lie in this region. 
This confirms the observations of Remark~\ref{rem:Vbounds}. 

\subsection{A Nonlinear Example}
\label{sec:nl_example}
We consider now the system
\begin{align*}
	\matr{c}{x_1 \\ x_2}_+ &= \matr{c}{2x_1 + (x_1-1)^3 + 1 + x_2 \\ x_2 + u}, \\
	\ell(x,u) &= x_2^2 + x_1 + (x_1-1)^3 + 1 + x_2 - u, \\
	x_1&\in[-0.5,0.5], \quad x_2\in[-10,10], \quad u\in[-10,10].
\end{align*}
We computed value functions $V_+$ and $V_-$ by means of dynamic programming on a discretized state and control space with a uniform grid. In order to compute $V_-$ we first computed $f^{-1}$ such that $x=f^{-1}(x_+,u)$, with $x_+=f(x,u)$, and then solved using dynamic programming the OCP formulated backward in time, consistently with the definition of $V_-$. 

\begin{figure}[t]
	\includegraphics[width=\linewidth,clip,trim=0 0 0 0]{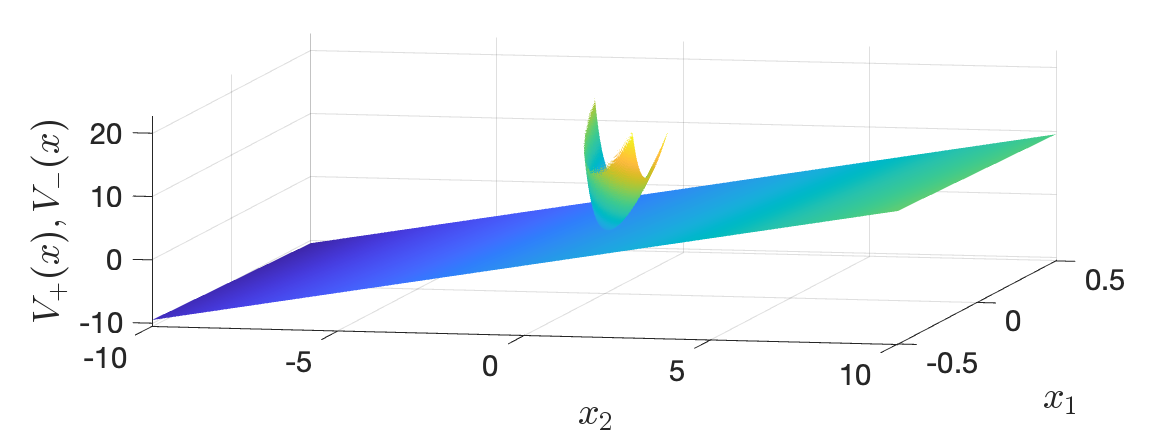}
	\includegraphics[width=\linewidth,clip,trim=0 0 0 0]{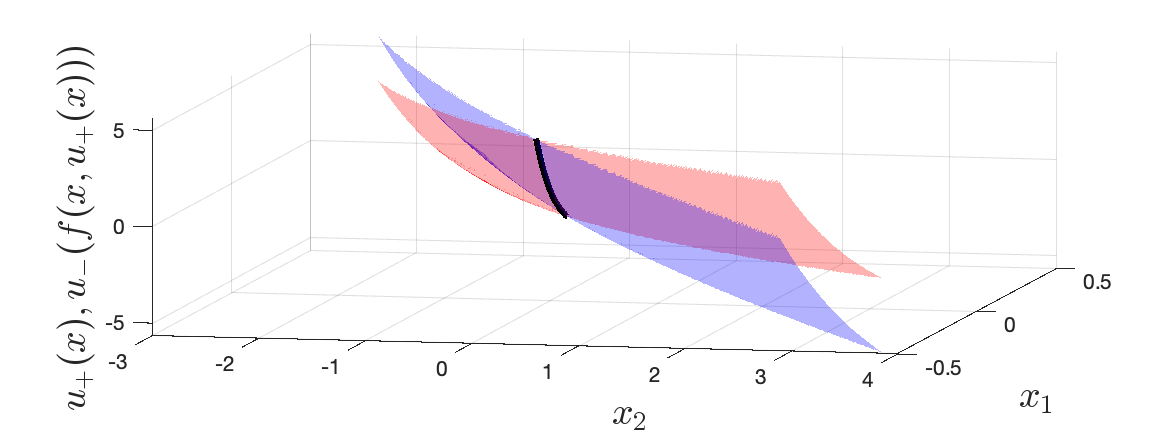}
	\includegraphics[width=\linewidth,clip,trim=0 0 0 0]{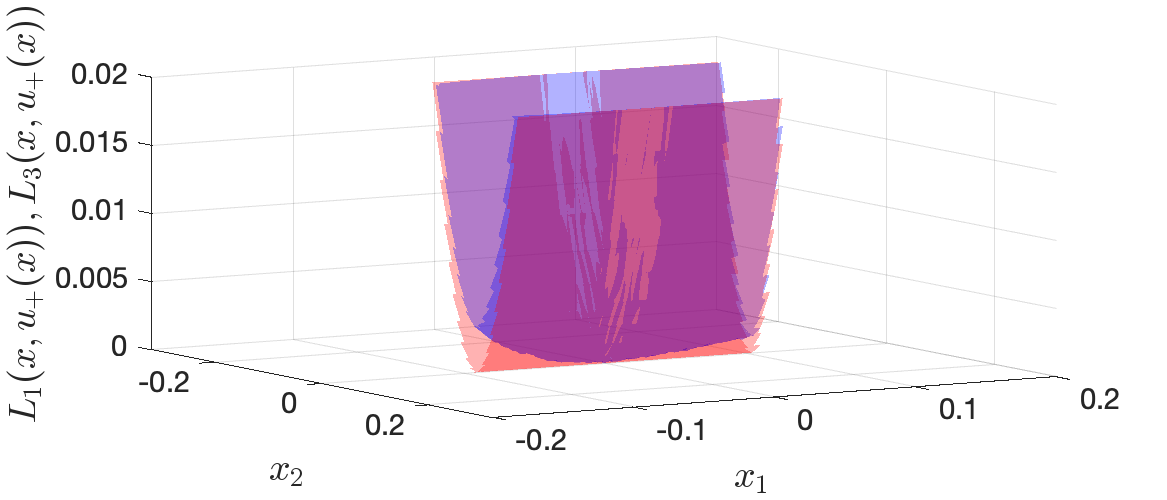}
	\caption{Top plot: $V_+(x)$, $V_-(x)$; center plot: $u_+(x)$, $u_-(f(x,u_+(x)))$, the black line highlights the case in which $u_+(x)=u_-(f(x,u_+(x)))$; bottom plot: rotated costs $L_1(x,u_+(x))$ (red) and $L_3(x,u_+(x))$ (blue) from Theorem~\ref{thm:2str_diss_implies_str_diss}.}
	\label{fig:nonlinear}
\end{figure}
In Figure~\ref{fig:nonlinear} one can see that $V_+(x)>V_-(x)$ for all $x\neq0$ and, in this case, $\mathcal{X}_+\subset\mathcal{X}_-$, such that we do not need to resort to $V_\ominus$ for this example. In the center plot one can see that there exist states $x\neq0$ for which $u_+(x)=u_-(f(x,u_+(x)))$, highlighted by a black line; in particular, in this example this happens for $x_2=0$. This has an important impact on some of the proofs we derived above, as it introduces some additional complexity that could be avoided in case $u_+(x)\neq u_-(f(x,u_+(x)))$ for all $x\in\mathcal{X}_+$, see, e.g., Remark~\ref{rem:discuss_main_thm}.   
In the bottom plot, we display the rotated costs from Theorem~\ref{thm:2str_diss_implies_str_diss}. One can observe that $L_1(x,u)\geq0$ by construction, but $L_1(x,u)\not>0$ for all $x\neq0$, as $L_1(x,u_+(x))$ is zero by construction for $u_+(x)=u_-(f(x,u_+(x)))$. This issue makes the additional step of introducing $L_3$ necessary in the proof of Theorem~\ref{thm:2str_diss_implies_str_diss}. Indeed, $L_3(x,u_+(x))=0 \implies x=0$ and $L_3(x,u_+(x))>0$ for all $x\neq0$. However, by using $L_3$ we are unable to prove that two-storage strict dissipativity implies strict dissipativity, as in fact we have verified numerically that $L_3$ does become negative for some $x$ and $u\neq u_+(x)$.

In order to discuss the finite horizon case, we consider two options for the terminal cost: $V_\mathrm{f}^1(x):=V_\ominus(x)+r\|x\|_2^2$; and  $V_\mathrm{f}^2(x):=V_\ominus(x)+r(V_+(x)-V_\ominus(x))$; both with different values of parameter $r$. In Figure~\ref{fig:nonlinear_mpc} we display in the top plot the difference $\displaystyle\max_x|V_N(x)-V_+(x)|$ to illustrate convergence; and in the bottom plot the smallest prediction horizon yielding stability, which we denote as $N_\mathrm{s}$ and which we evaluate empirically by performing closed-loop simulations starting from all points on the grid used for dynamic programming corresponding to finite values of $V_+$. One can observe that for $V_\mathrm{f}^2(x)$ with $r=1$ the terminal cost is $V_+$, such that the value function $V_N$ matches $V_+$ for any $N\geq1$ (modulo numerical inaccuracies). One can also see that for $V_\mathrm{f}^2$ we have $N_\mathrm{s}=1$ for all considered values of $r$. For $V_\mathrm{f}^1$, the convergence metric we chose shows a similar pattern as for $V_\mathrm{f}^2$ (with the exception of the special case $r=1$). However, $N_\mathrm{s}$ is larger in this case. Nevertheless, as the theory predicts, all chosen terminal costs do yield asymptotic stability for $N$ long enough and $V_N$ converges to $V_+$ as $N\to\infty$.

\begin{figure}
	\includegraphics[width=\linewidth,clip,trim=0 0 0 15]{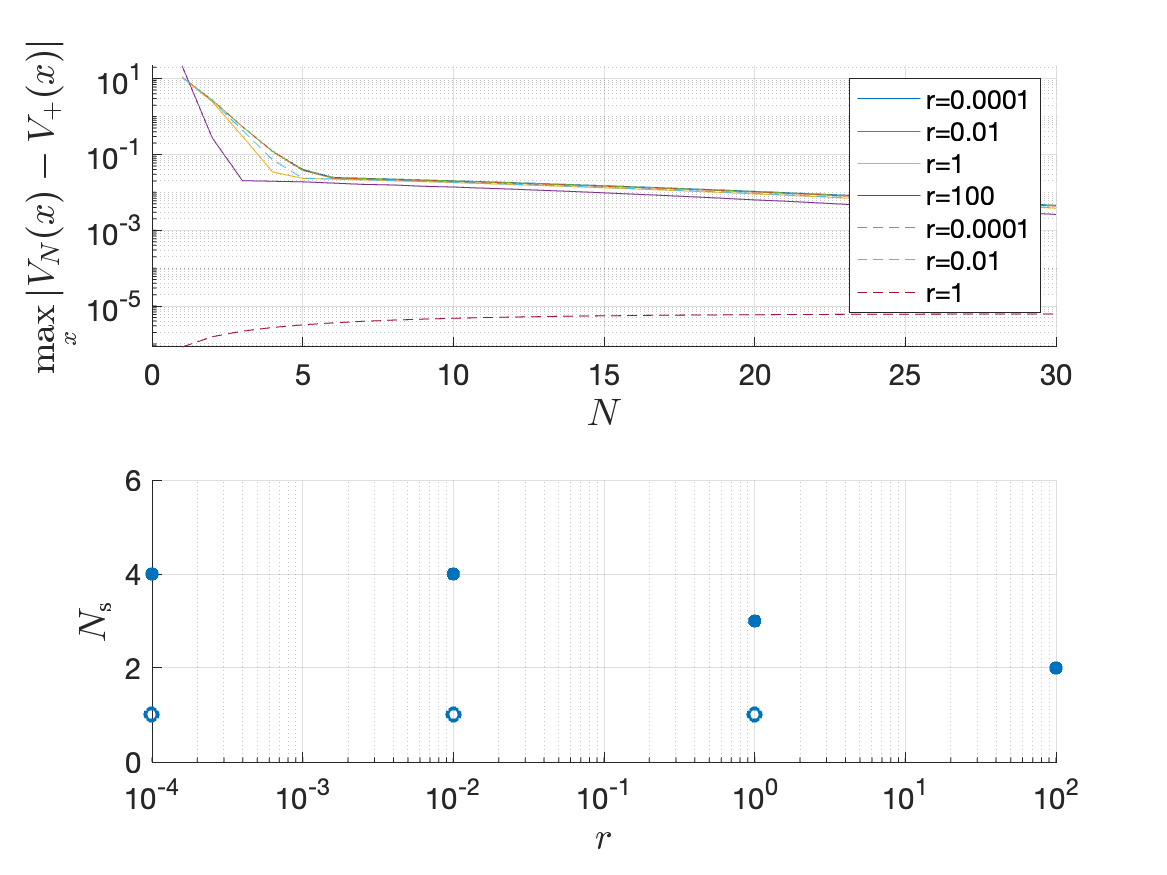}
	\caption{Top plot: convergence of $V_N$ to $V_+$ for $V_\mathrm{f}^1(x)$ (continuous lines) and $V_\mathrm{f}^2(x)$ (dashed lines) for different values of $r$; bottom plot: minimum stabilizing prediction horizon $N_\mathrm{s}$ for $V_\mathrm{f}^1(x)$ (dots) and $V_\mathrm{f}^2(x)$ (circles).}
	\label{fig:nonlinear_mpc}
\end{figure}

\section{Conclusions}
\label{sec:conclusions}
In this paper we have proposed two-storage strict dissipativity, a novel concept which, though apparently rather different from standard strict dissipativity, is actually very similar to it and allows us to prove strong results concerning the nature of dissipativity, its necessity and sufficiency for asymptotic stability and the construction of suitable terminal conditions for finite-horizon formulations. In particular, we have proven that the new concept is not a stronger assumption than the standard one, as it is implied by the latter. 

Future work will consider the extension of these results to other settings, including the periodic, discounted, and stochastic settings.

\section*{Acknowledgments}
The author would like to thank S\'ebastien Gros, Timm Faulwasser, Mario E. Villanueva, Boris Houska, Matthias M\"uller, Moritz Diehl, and especially Lars Gr\"une for the fruitful discussions on dissipativity and economic MPC.

\bibliographystyle{IEEEtran}
\bibliography{bibliography}

\end{document}